\RequirePackage{fix-cm}
\documentclass[american,english]{article}
\usepackage{mathptmx}
\usepackage[T1]{fontenc}
\usepackage[latin9]{inputenc}
\usepackage{amsthm}
\usepackage{amsmath}
\usepackage{amssymb}
\usepackage{fixltx2e}
\usepackage{esint}
\usepackage[all]{xy}
\PassOptionsToPackage{normalem}{ulem}
\usepackage{ulem}

\makeatletter
\numberwithin{equation}{section}
  \theoremstyle{plain}
  \newtheorem{thm}{\protect\theoremname}[section]
  \theoremstyle{definition}
  \newtheorem{example}{\protect\examplename}[section]
  \theoremstyle{plain}
  \newtheorem{cor}{\protect\corollaryname}[section]
  \theoremstyle{definition}
  \newtheorem*{example*}{\protect\examplename}
  \theoremstyle{plain}
  \newtheorem{lem}{\protect\lemmaname}[section]
  \theoremstyle{remark}
  \newtheorem{rem}{\protect\remarkname}[section]
  \newcounter{casectr}
  \newenvironment{caseenv}
  {\begin{list}{{\itshape\ \protect\casename} \arabic{casectr}.}{%
   \setlength{\leftmargin}{\labelwidth}
   \addtolength{\leftmargin}{\parskip}
   \setlength{\itemindent}{\listparindent}
   \setlength{\itemsep}{\medskipamount}
   \setlength{\topsep}{\itemsep}}
   \setcounter{casectr}{0}
   \usecounter{casectr}}
  {\end{list}}
  \theoremstyle{definition}
  \newtheorem{defn}{\protect\definitionname}[section]

\makeatother

\usepackage{babel}
  \addto\captionsamerican{\renewcommand{\casename}{Case}}
  \addto\captionsamerican{\renewcommand{\definitionname}{Definition}}
  \addto\captionsamerican{\renewcommand{\examplename}{Example}}
  \addto\captionsamerican{\renewcommand{\lemmaname}{Lemma}}
  \addto\captionsamerican{\renewcommand{\remarkname}{Remark}}
  \addto\captionsenglish{\renewcommand{\casename}{Case}}
  \addto\captionsenglish{\renewcommand{\definitionname}{Definition}}
  \addto\captionsenglish{\renewcommand{\examplename}{Example}}
  \addto\captionsenglish{\renewcommand{\lemmaname}{Lemma}}
  \addto\captionsenglish{\renewcommand{\remarkname}{Remark}}
  \providecommand{\casename}{Case}
  \providecommand{\definitionname}{Definition}
  \providecommand{\examplename}{Example}
  \providecommand{\lemmaname}{Lemma}
  \providecommand{\remarkname}{Remark}
\addto\captionsamerican{\renewcommand{\corollaryname}{Corollary}}
\addto\captionsamerican{\renewcommand{\theoremname}{Theorem}}
\addto\captionsenglish{\renewcommand{\corollaryname}{Corollary}}
\addto\captionsenglish{\renewcommand{\theoremname}{Theorem}}
\providecommand{\corollaryname}{Corollary}
\providecommand{\theoremname}{Theorem}

\begin{document}

\title{Spherical Functions on 2-adic Ramified Hermitian Spaces}

\date{Dror Ozeri}

\maketitle
\begin{center}
Ozeri@tx.technion.ac.il
\par\end{center}

\begin{center}
Department of Mathematics, Technion - Israel Institute of Technology,
Haifa 32000
\par\end{center}

~
\begin{abstract}
Y. Hironaka introduced the spherical functions on the p-adic space
of Hermitian matrices. For the space of $2\times2$ Hermitian matrices,
we complete Hironaka's work by also considering the case of a wildly
ramified quadratic extension. We compute the spherical functions explicitly
and obtain the functional equations.
\end{abstract}
\thispagestyle{empty}

~

~

~

~

\tableofcontents{}

~

~

~

~

~

~

\setcounter{page}{1}

~

~

~

~

\section{Notation}

~
\begin{itemize}
\item $E$ - Local p-adic field.
\item $E^{*}=E-\{0\}.$
\item $O_{E}$ The ring of integers of $E$.
\item $O_{E}^{*}$ - The units of $O_{E}.$
\item $\pi$ - A uniformizer of $E.$
\item $\overline{E}$ - The residue field of the field $E.$
\item $|\,|_{E}$ - The p-adic absolute value on $E$ normalized such that
$|\pi|_{E}=|\overline{E}|^{-1}.$
\item $v_{E}$ - The valuation map corresponding to $|\cdot|_{E}.$
\item $F$ - A quadratic extension of $E.$ 
\item $N=N_{F/E},Tr=Tr_{F/E}$ - The Norm and Trace maps.
\item $\varpi$ - A uniformizer of $F.$
\item $q=\#\overline{F}.$
\item $|\,|=|\,|_{F}$ - The p-adic absolute value of $F$ normalized by
$|\varpi|=q^{-1}.$
\item $v_{F}$ - The valuation map corresponding to $|\cdot|_{F}.$
\item $x\mapsto\overline{x}$ - The conjugation map of the field extension.
\item $s=v_{F}[\overline{(\varpi)}\cdot\varpi^{-1}-1]$ - The number is
defined only if $F/E$ is ramified.
\item $l=[\frac{s}{2}]$ - The largest integer $k$ such that $k\leq\frac{s}{2}$.
\item $\Delta\in1+\pi^{s}O_{E}\,,\rho\in O_{E}^{*}$ - Fixed elements such
that $\Delta=1+\rho\pi^{s}\notin N(F).$ 
\item $\chi^{*}$ - The non trivial character of $E^{*}/N(F^{*}).$ 
\item $1_{A}$ - The characteristic function of a set $A$.
\item $G$=$GL_{2}(F).$ 
\item $K=GL_{2}(O_{F}).$
\item For a group $A$ and a normal subgroup $B$ we denote by $[x]_{A/B}$
to be the coset $xB$ of an element $x\in A$.
\end{itemize}

\section{Introduction\label{sec:Introduction}}

~

Let $E$ be a p-adic local field and $F$ a quadratic extension of
$E$. Let $O_{F}$ be the ring of integers of $F$, $G=GL_{2}(F)$,
and $K=GL_{2}(O_{F})$. Set $A^{*}=\overline{A}^{t}$ . Let $X=\{A\in GL_{2}(F)|\,\, A^{*}=A\}$.
We note that the group $G$ acts on $X$ by $g\cdot x=gxg^{*}$.

~

We denote by $C^{\infty}(K\backslash X)$ the space of complex-valued
$K$-invariant functions on $X$ and $S(K\backslash X)\subseteq C^{\infty}(K\backslash X)$
the subspace of all compactly supported $K$ invariant functions on
$X$. Let $\mathcal{H}(G,K)$ be the Hecke algebra of $G$ with respect
to $K$ : the space of all compactly supported $K$-bi-invariant complex
valued functions on $G$.

~

Let $dg$ be the Haar measure on $G$ normalized by $\underset{G}{\int}1_{K}dg=1$
. The algebra $\mathcal{H}(G,K)$ acts on $C^{\infty}(K\backslash X)$
and on $S(K\backslash X)$ by the convolution product: 
\[
f\cdot\phi(x)=\underset{G}{\int}f(g)\phi(g^{-1}\cdot x)dg.
\]

Under this action we call $\phi\in C^{\infty}(K\backslash X)$ a \textit{spherical
function} on $X$ if $\phi$ is a common $\mathcal{H}(G,K)$-eigenfunction.

~

Recall that in a quadratic extension we have the following maps: $N:F^{*}\to E^{*},Tr:F\to E$
, where:
\[
N(x)=x\overline{x},\,\, Tr(x)=x+\overline{x}.
\]

~

~

Let $(s_{1},s_{2})\in\mathbb{C}^{2}$ , $\chi_{1},\chi_{2}$ characters
of $E^{*}/N(F^{*})$ , $x\in X$ . Hironaka introduced in H{[}1-4{]}
the following function :
\[
L(x,\chi_{1},\chi_{2},s_{1},s_{2})=\underset{K'}{\int}\underset{i=1}{\overset{2}{\prod}}\chi_{i}(d_{i}(k\cdot x))|d_{i}(k\cdot x)|^{s_{i}}dk,
\]

where $dk$ is the Haar measure on $K$ normalized such that $\underset{K}{\int}dk=1$
,

$d_{1}(y)=y_{1,1},\,\, d_{2}(y)=det(y)$ and $K'=\{k\in K|\underset{i=1}{\overset{2}{\prod}|}d_{i}(k\cdot x)|\neq0\}$.

~

~

It is known that this integral converges absolutely for $Re(s_{1}),Re(s_{2})>0$
and admits a meromorphic continuation to a rational function in $q^{s_{1}},q^{s_{2}}$.

~

~

~

We transform the variables $s=(s_{1},s_{s})\in\mathbb{C}^{2}$ to
new variables $z=(z_{1},z_{2})$ by the following equations:
\[
\begin{array}{c}
s_{1}=z_{2}-z_{1}-\frac{1}{2}\\
s_{2}=-z_{2}+\frac{1}{4}.
\end{array}
\]

~

It is known that this function is indeed a spherical function \cite{key-1}
and for any $f\in\mathcal{H}(G,K)$ we have

\[
[f\cdot L(*,\chi_{1},\chi_{2},z)](x)=\tilde{f}(z)\times L(x,\chi_{1},\chi_{2},z),
\]

~

where $\tilde{f}(z)$ is defined to be the Satake transform:
\[
\tilde{f}(z)=\underset{G}{\int}f(g)\Phi_{2z}(g)dg
\]

~

and $\Phi_{2z}(g)=|a_{1}|^{2z_{1}-\frac{1}{2}}|a_{2}|^{2z_{2}+\frac{1}{2}}$
where $a_{1},a_{2}$ are determined by the Iwasawa decomposition:
\[
g=k\left(\begin{array}{cc}
a_{1}\\
 & a_{2}
\end{array}\right)\left(\begin{array}{cc}
1 & *\\
 & 1
\end{array}\right),\,\,\, k\in K.
\]

By abuse of notation we denote $L(x,\chi_{1},\chi_{2},z)=L(x,\chi_{1},\chi_{2},s(z))$.

~

If the field extension is ramified, it is known that $s=v_{F}[\varpi^{-1}\overline{(\varpi)}-1]$
is an invariant of the field extension $F/E$ (see \cite{key-9} Ch3
p. 70 ) .

~

Hironaka computed $L(x,\chi_{1},\chi_{2},z)$ for the following cases
\cite{key-5,key-2}: 

~

The case where the field extension $F/E$ is unramified, the case
where the fields extension $F/E$ is ramified and $|2|=1$ (Tamely
Ramified), and the case where the field extension $F/E$ is ramified
, $|2|_{E}<1$ and $s=1,2$ (Wildly Ramified). Hironaka also proved
that the spherical functions satisfy a functional equation \cite{key-3}
and used it to compute the spherical functions for general $G=GL_{n}$
in Case 1.

~

In this paper we complete Hironaka\textquoteright{}s calculation of
Case 3 for the case of general uniformizer $\pi$ of $E$ and general
$s$. We compute the spherical functions defined above and provide
the general functional equations. Our computation is by brute force
and applies properties of the Norm and Trace maps in a quadratic extension
of local fields.

~

Chapter 3 of this work will be dedicated to the summary of the facts
we need from the local field theory. In Chapter 4 we compute useful
(and interesting) p-adic integrals that we use in Chapters 7 and 8
. In Chapter 5 we explicate convenient representatives of $K\backslash X$
, following Jacobowitz \cite{key-10} . In Chapter 6 we state the
main theorems of this work regarding the computation of the spherical
functions and the functional equation. In Chapters 7 and 8 we prove
the main theorem of Chapter 6. 

~

It is our hope that the functional equations can be used in the future
to calculate the spherical functions for general $GL_{n}(F)$ case
as was shown in \cite{key-4,key-7} with the Casselman-Shalika basis
method \cite{key-6}.

~

Spherical functions on the Hermitian spaces $X$ are related to the
concept of local densities \cite{key-4,key-5,key-8} and calculation
of the spherical functions for $GL_{n}$ can be used to calculate
the local densities , which are important in several aspects.

~

Integrals of the form $\underset{O_{F}}{\int}|a+Tr(bx)+cN(x)|_{F}^{z}dx$
appear naturally in many places. It is hoped that their computation
will have further applications

~

I would like to thank my advisors Dr. Omer Offen and Ass.~Prof.~Moshe
Baruch from the Technion for introducing me to this subject and for
their endless help and support for the last 2 years.

~

~

~

\section{\label{sec:Local-Fields-Properties:}Local Fields Properties:}

~

Let $E$ be a p-adic local field, and $F/E$ a quadratic extension.
We normalize the absolute value on $F$ by $|\varpi|=[\#\bar{F}]^{-1}$
.We denote by $v_{F}$, $v_{E}$ the corresponding valuation maps.

~

We recall without proofs facts and properties from local class field
theory, proofs can be found in (\cite{key-9}, Chapter 3). 

~

We distinguish between the following cases of extension:

~

\textbf{\emph{\large Case 1}}: \textbf{\uline{\large ~~~~The
unramified case}}{\large \par}

~

$F/E$ is an unramified extension. We denote $|\overline{F}|=q^{2},\,\,\,|\bar{E}|=q.$
For convenience, we take $\pi=\varpi.$

~

\textbf{\emph{\large Case 2}}: \textbf{\large ~~~}\textbf{\uline{\large The
tamely ramified case}}{\large \par}

~

$F/E$ is totally ramified and $Char(\bar{E})\neq2.$ We denote $q=|\bar{E}|=|\overline{F}|$.
For convenience, we take $\pi=N(\varpi).$

~

\textbf{\emph{\large Case 3}}:~~~\textbf{\uline{\large ~The
wildly ramified case}}{\large \par}

~

$F/E$ is totally ramified and $Char(\bar{E})=2.$ We denote $q=|\bar{E}|=|\overline{F}|$.
For convenience, we take $\pi=N(\varpi).$

~

~
\begin{thm}
\label{theorem of hasse}Let $F/E$ be a wildly ramified extension,
then :

~

1. The number $s=v_{F}[\varpi^{-1}\overline{\varpi})-1]$ does not
depend on the choice of the uniformizer $\varpi.$ 

2. $0<s\leq2v_{E}(2).$

3. If $s$ is even then $s=2v_{E}(2).$

~
\end{thm}
Proof of Theorem \ref{theorem of hasse} could be found in \cite{key-9}
,p. 75.

~

Theorem \ref{theorem of hasse} motivates us to distinguish between
two types of extensions:

~

~

~

\textbf{\uline{\large Ramified Prime (RP) }}{\large \par}

~

The extension $F/E$ is wildly ramified extension and the invariant
$s$ is even. It can be shown \cite{key-9} that if $F/E$ is RP then
$F=E(\sqrt{\pi'})$ , where $\pi'$ is some uniformizer of $E$ . 

~

~

~

~

~

~

\textbf{\uline{\large Ramified Unit (RU) }}{\large \par}

~

The extension $F/E$ is wildly ramified extension and the invariant
$s$ is odd. Again, it can be shown that if $F/E$ is RU then $F=E(\sqrt{1+\delta\pi^{2k+1}})$
for some $\delta\in O_{E}^{*}$ . One can take a uniformizer on $F$
to be $\varpi'=\frac{1+\sqrt{1+\delta\pi^{2k+1}}}{\pi^{k}}$ ( note
that $|\varpi'\overline{\varpi'}|_{E}=q^{-1}$). A quick calculation
shows that $s=2v_{E}(2)-(2k+1)>0.$

~
\begin{example}
\label{example RP}Consider the extension $F/E$ where $F=Q_{2}(\sqrt{2})$
, $E=Q_{2}$.

Since $|2|_{F}=|\sqrt{2}\cdot\sqrt{2}|_{F}=2^{-1}\Longrightarrow|\sqrt{2}|_{F}=\sqrt{2}^{-1}$
, but $2$ is known to be the uniformizer of $Q_{2}$, hence this
extension is ramified. $(|\varpi|_{F}>|\pi|_{F})$. We take $\varpi=\sqrt{2}$.
Note that $s=v_{F}(\frac{\sigma(\sqrt{2})}{\sqrt{2}}-1)=v_{F}(\frac{-\sqrt{2}}{\sqrt{2}}-1)=v_{F}(-1-1)=v_{F}(-2)=2$
. Therefore this extension is a RP extension.
\end{example}
~
\begin{example}
\label{example RU}Consider the extension $F/E$ where $F=Q_{2}(\sqrt{-5})$
, $E=Q_{2}$.
\end{example}
By one of the definitions of the absolute value $|\cdot|_{F}:$ 
\[
|a+\sqrt{-5}b|_{F}=\sqrt{|a^{2}+5b^{2}|_{E}},\,\,\,\, a,b\in E.
\]
 Therefore: $|1+\sqrt{-5}|_{F}=\sqrt{|6|_{E}}=\sqrt{2}^{-1}$. So
$F/E$ is ramified extension. 

One can take $\varpi=1+\sqrt{-5}.$ We calculate: 
\[
s=v_{F}(\frac{1-\sqrt{-5}}{1+\sqrt{-5}}-1)=v_{F}(\frac{-2\sqrt{-5}}{1+\sqrt{-5}})=v_{F}(\frac{-2}{1+\sqrt{-5}})=v_{F}(2)-v_{F}(1+\sqrt{-5})=2-1=1.
\]
 We conclude that $F/E$ is a ramified unit extension. Note that $F=E(\sqrt{1+(-3)\cdot2^{1}}$.

\subsection{The Norm\label{sub:The-Norm}}

~

For $i>0$ ,we denote by $\lambda_{i,F}$ (resp. $\lambda_{i,E}$)
the natural map $\lambda_{i,F}:\,(1+\pi^{i}O_{F})\rightarrow\overline{F}$

\[
\lambda_{i,F}(x)=\varpi^{-i}(x-1)\,\,\, mod\,\,\varpi O_{F}
\]

and 

\[
\lambda_{i,E}(x)=\pi^{-i}(x-1)\,\,\, mod\,\,\pi O_{F}.
\]

~

Denote the residue map by $\lambda_{0,K}:O_{K}\rightarrow\overline{K}$,
$(K=E,F)$ .

~

~

~

We summarize the properties of the norm by the following diagrams:

~

\textbf{\uline{Case 1-Unramified :}}

~

The following diagrams commute:

~

\[
\xymatrix{F^{*}\ar[d]_{N_{F/E}}\ar[r]^{v_{F}} & \mathbb{Z}\ar[d]^{\times2}\\
E^{*}\ar[r]^{v_{E}} & \mathbb{Z}
}
\]
\begin{equation}
\xymatrix{O_{F}\ar[d]_{N_{F/E}}\ar[r]^{\lambda_{0,F}} & \bar{F}\ar[d]^{N_{\overline{F}/\overline{E}}}\\
O_{E}\ar[r]^{\lambda_{0,E}} & \bar{E}
}
\label{eq:3.1-1}
\end{equation}
~~~~~~
\[
\,\,\,\,\,\,\,\,\,\,\,\,\xymatrix{1+\varpi^{i}O_{F}\ar[d]_{N_{F/E}}\ar[r]^{\lambda_{i,F}} & \bar{F}\ar[d]^{Tr_{\overline{F}/\overline{E}}}\\
1+\pi^{i}O_{E}\ar[r]^{\lambda_{i,E}} & \bar{E}
}
,\, i\geq1
\]

~

\textbf{\uline{Case 2- Tamely ramified}}

~

\[
\xymatrix{E^{*}=F^{*}\ar[d]_{N_{F/E}}\ar[r]^{v_{F}} & \mathbb{Z}\ar[d]^{id}\\
E^{*}\ar[r]^{v_{E}} & \mathbb{Z}
}
\]
\[
\xymatrix{O_{F}\ar[d]_{N_{F/E}}\ar[r]^{\lambda_{0,F}} & \bar{E}=\bar{F}\ar[d]^{x\mapsto x^{2}}\\
O_{E}\ar[r]^{\lambda_{0,E}} & \bar{E}
}
\]
\[
\,\,\,\,\,\xymatrix{1+\varpi^{2i}O_{F}\ar[d]_{N_{F/E}}\ar[r]^{\lambda_{2i,F}} & \bar{E}=\overline{F}\ar[d]^{\times2}\\
1+\pi^{i}O_{E}\ar[r]^{\lambda_{i,E}} & \bar{E}
}
\]

~

And $N(1+\varpi^{2i-1}O_{F})=N(1+\varpi^{2i}O_{F})$.

~

~

~

~

\textbf{\uline{Case 3- wildly ramified}}

~

Let $\eta\in O_{F}^{*}$ to be such that $\frac{\overline{\varpi}}{\varpi}=1+\eta\varpi^{s}$
. 

Denote $\kappa=\lambda_{0,F}(\eta)$.

~

~

The following diagrams commute:

\[
\xymatrix{F^{*}\ar[d]_{N_{F/E}}\ar[r]^{v_{F}} & \mathbb{Z}\ar[d]^{id}\\
E^{*}\ar[r]^{v_{E}} & \mathbb{Z}
}
\]

\begin{equation}
\xymatrix{O_{F}\ar[d]_{N_{F/E}}\ar[r]^{\lambda_{0,F}} & \bar{E}=\bar{F}\ar[d]^{x\mapsto x^{2}}\\
O_{E}\ar[r]^{\lambda_{0,E}} & \bar{E}
}
\label{eq:3.1}
\end{equation}

~

~

For $i<s$:

\begin{equation}
\xymatrix{1+\varpi^{i}O_{F}\ar[d]_{N_{F/E}}\ar[r]^{\lambda_{i,F}} & \bar{E}=\overline{F}\ar[d]^{x\mapsto x^{2}}\\
1+\pi^{i}O_{E}\ar[r]^{\lambda_{i,E}} & \bar{E}
}
\label{eq:3.2}
\end{equation}

~

~

~~~
\begin{equation}
\xymatrix{1+\varpi^{s}O_{F}\ar[d]_{N_{F/E}}\ar[r]^{\lambda_{s,F}} & \bar{E}=\overline{F}\ar[d]^{x\mapsto x^{2}-\kappa\cdot x}\\
1+\pi^{s}O_{E}\ar[r]^{\lambda_{s,E}} & \bar{E}
}
\label{eq:3.3}
\end{equation}

~

(Note that the homomorphism $x\mapsto x^{2}-\kappa\cdot x$ is an
additive homomorphism with kernel of size 2)

~

~

~

~

~

For $j>0$

\begin{equation}
\xymatrix{1+\varpi^{s+2j}O_{F}\ar[d]_{N_{F/E}}\ar[r]^{\lambda_{s+2j,F}} & \bar{E}=\overline{F}\ar[d]^{\times\kappa}\\
1+\pi^{s+j}O_{E}\ar[r]^{\lambda_{i,E}} & \bar{E}
}
\,\label{eq:3.4}
\end{equation}

~

and $N(1+\varpi^{s+i}O_{F})=N(1+\varpi^{s+i+1}O_{F})$~~~if $i>0$
and $2\nmid i.$

~

Proofs for commutativity could be found in \cite{key-9} p. 68-73.

~

We will use throughout this paper the following corollaries ( Case
3):
\begin{cor}
\label{corollary norm 1}$N(1+\varpi^{s+1}O_{F})=1+\pi^{s+1}O_{E}$.\end{cor}
\begin{proof}
From \eqref{eq:3.4} we conclude that since the map $x\mapsto x\times\kappa$
is a surjective homomorphism between the additive groups $\overline{F}\to\overline{E}$
and form the commutativity of diagram \ref{eq:3.4} that for $i>s$
, every ball of radius $q^{-i}$ in $1+\pi^{s+1}O_{E}$ contains an
element that is a norm, thus $N(1+\varpi^{s+1}O_{F})$ is dense in
$1+\pi^{s+1}O_{E}$. From compactness of $1+\varpi^{s+1}O_{F}$ we
have $N(1+\varpi^{s+1}O_{F})=1+\pi^{s+1}O_{E}$ .
\end{proof}
~
\begin{cor}
$[1+\pi^{s}O_{E}\,:\, N(1+\varpi^{s}O_{F})]=2.$\end{cor}
\begin{proof}
Follows from commutativity of \eqref{eq:3.3} and the last corollary
\end{proof}
~
\begin{cor}
\label{corolary 3.3}Let $x,y\in\pi^{i}O_{E}^{*}$ , If $|x-y|_{E}<q^{-s-i}$
then $x\in N(F)\Longleftrightarrow y\in N(F)$ .\end{cor}
\begin{proof}
Since $\pi$ is a norm we can assume w.l.o.g that $i=0$. We have
from the last corollary and the commutativity of \eqref{eq:3.1-1}
that $N(O_{E}^{*})$ is a union of cosets of $1+\pi^{s+1}O_{E}$,
thus if $|x-y|_{E}<q^{-s}$ then they are in the same coset and we
have $x\in N(O_{F}^{*})\iff y\in N(O_{F}^{*})$.\end{proof}
\begin{example*}
Consider $F/E$ from example \ref{example RP}. Note that $\overline{F}=\overline{E}=F_{2}$
(The field with two elements) . Note that $s=2$, and $\kappa=1$
. we have:

\[
\xymatrix{1+2O_{F}\ar[d]_{N_{F/E}}\ar[r]^{\lambda_{2,F}} & F_{2}\ar[d]^{x\mapsto x^{2}-x}\\
1+4Z_{2}\ar[r]^{\lambda_{2,E}} & F_{2}
}
\]
Since the image of the map: $\psi:F_{2}\to F_{2}$ , $\psi(x)\mapsto x^{2}-x$
is $\{0\}$ we conclude from the commutativity of the diagram that:

\[
N(1+2O_{F})=1+8Z_{2}.
\]

\end{example*}

\subsection{The Trace\label{sub:The-Trace}}

~

In the wildly ramified case (Case 3) the following holds:

~

$Tr_{F/E}(\varpi^{i}O_{F})=\pi^{j(i)}O_{E}$ , $j(i)=s+1+[(i-1-s)/2]$. 

~

\selectlanguage{american}%
Proof could be found in \cite{key-9}, p. 71.

~

\selectlanguage{english}%
Explicitly,

~

\textbf{RP:}

$s=2l,$ $l=\nu_{E}(2)$

$Tr(\varpi^{2i}O_{F})=\pi^{l+i}O_{E}$

$Tr(\varpi^{2i-1}O_{F})=\pi^{l+i}O_{E}$

~

~

\textbf{RU:}

$s=2l+1$

\selectlanguage{american}%
$Tr(\varpi^{2i}O_{F})=\pi^{l+1+i}O_{E}$

$Tr(\varpi^{2i-1}O_{F})=\pi^{l+i}O_{E}$

~

\selectlanguage{english}%

\section{Useful Integrals}
\begin{lem}
\label{pushforword lemma}Let H and G be compact abelian topological
groups with Haar measures $\mu_{H},\mu_{G}$ and $\phi:H\rightarrow G$
a surjective (continuous) homomorphism, then the push-forward measure
$\mu_{H*}=\mu_{H}\circ\phi^{-1}$is an invariant Haar measure on $G$
and $\mu_{H*}=\frac{\mu_{H}(H)}{\mu_{G}(G)}\cdot\mu_{G}$.\end{lem}
\begin{proof}
Because the characters span a dense subset in $L^{1}(G,\mu_{G})$
it is sufficient to show that $\underset{G}{\int}\chi d\mu_{H*}=0$
for any non-trivial character $\chi$ . But $\underset{G}{\int}\chi d\mu_{H*}=\underset{H}{\int}\chi(\phi(x))d\mu_{H}(x)=0,$
since $\chi(\phi(x))$ is an non trivial character of $H$ . We obtain
the multiplicative factor between the measures from integrating over
the trivial character.\end{proof}
\begin{rem}
The conditions of the previous lemma can be further generalized. 
\end{rem}
~

~

Take $\mu_{E},\mu_{F}$ to be the Haar measures on $E,F$ normalized
by 
\[
\mu_{E}(O_{E})=\mu_{F}(O_{F})=1.
\]

~
\begin{lem}
\label{measure lemma}For $n\geq1$ 
\[
\mu_{F}[N^{-1}(1+\pi^{n}O_{E}^{*})]=\mu_{F}(N^{-1}[(1+\pi^{n}O_{E})-(1+\pi^{n+1}O_{E})])=
\]
 \end{lem}
\begin{caseenv}
\item $\frac{q^{2}-1}{q^{n+2}}$
\item $\frac{2(q-1)}{q^{n+1}}$
\item $\begin{cases}
\begin{array}{cc}
(q-1)q^{-(n+1)} & \,\, if\,\,\, n<s\\
(q-2)q^{-(s+1)} & \,\, if\,\,\, n=s\\
2(q-1)q^{-(n+1)} & \: if\,\,\, n>s
\end{array}\end{cases}$
\end{caseenv}
~
\begin{proof}
In Case 1, since $N(O_{F}^{*})=O_{E}^{*}$ we conclude that the push-forward
of the multiplicative Haar measure of $O_{F}^{*}$ is an multiplicative
Haar measure of $O_{E}^{*}$ . In $O_{F}^{*}$ the additive and multiplicative
measures coincide, so we get that 
\[
\mu_{F*}=\frac{\mu_{F}(O_{F}^{*})}{\mu_{E}(O_{E}^{*})}\mu_{E}=\frac{1-q^{-2}}{1-q^{-1}}\mu_{E}=(1+q^{-1})\mu_{E}.
\]
Since $\mu_{E}(1+p^{n}O_{E}^{*})=q^{-n}(1-q^{-1})$, we get that:

\[
\mu_{F*}(1+\pi^{n}O_{E}^{*})=(1+q^{-1})\cdot q^{-n}(1-q^{-1})=\frac{q^{2}-1}{q^{n+2}}
\]

~

In Case 2, we have $N(1+\varpi O)=1+\pi O_{E}$ . Since this case
is ramified, we have that $\mu_{F}(O_{F}^{*})=\mu_{E}(O_{E}^{*})$.
So $\mu_{F}[N|_{1+\varpi O_{F}}^{-1}(1+\pi^{n}O_{E}^{*})]=q^{-n}(1-q^{-1})$
($N|_{1+\varpi O_{F}}$ is the restriction of the norm to $1+\varpi O_{F}$)
. But since the norm induces the square map between the groups
\[
\widetilde{N}:\frac{O_{F}^{*}}{1+\varpi O_{F}}\rightarrow\frac{O_{E}^{*}}{1+\pi O_{E}}.
\]
 The size of the kernel of the this map is 2. We have that :
\[
N^{-1}(1+\pi^{n}O_{E}^{*})=a\cdot N|_{1+\varpi O_{F}}^{-1}(1+\pi^{n}O_{E}^{*})\dot{\cup}N|_{1+\varpi O_{F}}^{-1}(1+\pi^{n}O_{E}^{*}),
\]
where $a$ is a non trivial representative of the kernel of $\widetilde{N}$.
So:

~

\[
\mu_{F*}(1+\pi^{n}O_{E}^{*})=2\mu_{F}[N|_{1+\varpi O_{F}}^{-1}(1+\pi^{n}O_{E}^{*})]=\frac{2(q-1)}{q^{n+1}}.
\]

~

~

In Case 3 we have $N(1+\varpi^{s+1}O_{F})=1+\pi^{s+1}O_{E}$ . So
if we restrict the norm to $1+\varpi^{s+1}O_{F}$ we have $\mu_{*}=\mu_{E}$.

~

From the diagrams of Section \ref{sub:The-Norm} that the norm map
induces an homomorphism of the finite groups: 
\[
\widetilde{N}:\frac{1+\varpi O_{F}}{1+\varpi^{s+1}O_{F}}\to\frac{1+\pi O_{E}}{1+\pi^{s+1}O_{E}}.
\]
 The size of the kernel of this map is 2.

~

We have that for $n>s$:

\[
N^{-1}(1+\pi^{n}O_{E}^{*})=a[N|_{1+\varpi^{s+1}O_{F}}^{-1}(1+\pi^{n}O_{E}^{*})]\dot{\cup}N|_{1+\varpi^{s+1}O_{F}}^{-1}(1+\pi^{n}O_{E}^{*}),
\]
($N|_{1+\varpi^{s+1}O_{F}}$ is the restriction of the norm to $1+\varpi^{s+1}O_{F}$),
where $a$ is a representative element of the kernel of $\widetilde{N}$.
But since $N|_{1+\varpi^{s+1}O_{F}}$ is surjective on $1+\pi^{s+1}O_{E}$
we can use Lemma \ref{pushforword lemma} to obtain:

\[
\mu_{F}[N^{-1}(1+\pi^{n}O_{E}^{*})]=2\mu_{E}(1+\pi^{n}O_{E}^{*})=2(1-q^{-1})q^{-n}.
\]

~

For $n=s$ , we can conclude from \eqref{eq:3.4} that $N^{-1}(1+\pi^{s}O_{E}^{*})=\underset{i=1}{\overset{q-2}{\cup}}a_{i}(1+\varpi^{s+1}O_{F}),$
where the $a_{i}$ are representatives of the cosets that are the
preimages of $\widetilde{N}$ . So $\mu_{F}(N^{-1}(1+\pi^{s}O_{E}^{*}))=(q-2)q^{-(s+1)}$
.

~

For $n<s$ the norm induces an isomorphism between the finite groups
$\frac{O_{F}^{*}}{1+\varpi^{n+1}O_{F}}$ and $\frac{O_{E}^{*}}{1+\pi^{n+1}O_{E}},$
therefore $\mu_{F}(N^{-1}(1+\pi^{n}O_{E}^{*}))=\mu_{E}(1+\pi^{n}O_{E}^{*})=(q-1)q^{-(n+1)}$
, altogether we obtain the result above.

~

~\end{proof}
\begin{cor}
(case 1) 
\end{cor}
~

\begin{equation}
\underset{x\in O_{F}}{\int}|1+N(x)|^{s_{1}}dx=\frac{q^{2}-q-1}{q^{2}}+\frac{q^{2}-1}{q^{2}}\left(\frac{q^{-s_{1}-1}}{1-q^{-s_{1}-1}}\right)\label{I1}
\end{equation}

\begin{proof}
Note that $-1$ is a norm, hence $-1=\zeta\cdot\overline{\zeta}$
for some $\zeta\in O_{F}^{*}$. We substitute $x=\zeta y$:

\[
\underset{x\in O_{F}}{\int}|1+N(x)|^{s_{1}}dx=\underset{y\in O_{F}}{\int}|1-N(y)|^{s_{1}}dy=\underset{x\in O_{F}}{\int}|(-1)+N(y)|^{s_{1}}dy.
\]
The integrand is equal to $1$ for values of $y$ such that $N(y)\neq-1\,\,\, mod\,\,\,\pi O_{E}$.
From \eqref{eq:3.1-1} $q+1$ (the size of the kernel of $N_{\overline{F}/\overline{E}}$)
cosets of $O_{F}^{*}$ are map to the coset of $-1$. It is easy to
see that the integrand is also equal to $1$ in the set $\varpi O_{F}.$
Therefore: 
\[
\mu_{F}(-1+O_{E}^{*})=\underset{\mu_{F}([O_{F}^{*}-N^{-1}(-1+\pi O_{E})]}{\underbrace{\frac{(q^{2}-1)-q-1}{q^{2}}}}+\underset{\mu_{F}(\varpi O_{F})}{\underbrace{\frac{1}{q^{2}}}}=\frac{q^{2}-q-1}{q^{2}}.
\]

~

It follows immediately from Lemma \ref{measure lemma} that the integral
is equal to the following geometric sum:

\[
\underset{x\in O_{F}}{\int}|1+N(x)|^{s_{1}}dx=\frac{q^{2}-q-1}{q^{2}}+\underset{1\leq n}{\sum}\frac{q^{2}-1}{q^{2}}q^{-ns_{1}-n}=\frac{q^{2}-q-1}{q^{2}}+\frac{q^{2}-1}{q^{2}}\left(\frac{q^{-s_{1}-1}}{1-q^{-s_{1}-1}}\right)
\]

\end{proof}
~

~

For the next  lemmas, let $\chi^{*}$ be a non trivial character of
$E^{*}$ that is trivial on $N(F)$ . We will define $\chi^{*}(0)=0$. 

~

~
\begin{lem}
\label{charchter lemma 1}(Case 3 ) Let $0<m<s+1$ , $\theta\in O_{E}^{*}$
,then $\underset{x\in O_{F}}{\int}\chi^{*}(1+\theta\pi^{m}N(x))dx=0$\end{lem}
\begin{proof}
First, we show that the integrand is constant on the (additive) cosets
of group $\frac{O_{F}}{\varpi^{s+1-m}O_{F}}$ : 

If $x\in\varpi^{s+1-m}O_{F}$ then $\chi^{*}(1+\theta\pi^{m}N(x))=1$
by corollary \ref{corollary norm 1}. 

Let $x,y\in O_{F}^{*}$ . Assume that $x=yu$ ,where $u\in1+\varpi^{s+1-m}O_{F}$
(that is, they are in the same coset of $\frac{O_{F}}{\varpi^{s+1-m}O_{F}}$).
We know (see Section \ref{sub:The-Norm}) that $N(u)\in1+\pi^{s+1-m}O_{E}$.
So $N(u)=1+\pi^{s+1-m}\nu,\,\,\nu\in O_{F}$ . Then : 
\[
1+\theta\pi^{m}N(x)=1+\theta\pi^{m}N(y)N(u)=1+\theta\pi^{m}N(y)+\delta,\,\,\,\,\delta=\pi^{s+1}\theta\nu.
\]

Recall from Section \ref{sub:The-Norm} , that for $g_{1},g_{2}\in O_{E}^{*}$
if $|g_{1}-g_{2}|<q^{-s}$ ,then $g_{1}\in N(F)$ if and only if $g_{2}\in N(F)$
, since $|\delta|<q^{-s}$ we conclude that:

\[
\chi^{*}(1+\theta\pi^{m}N(x))=\chi^{*}(1+\theta\pi^{m}N(y)).
\]

We conclude that the integrand is constant on the cosets of $\frac{O_{F}}{\varpi^{s+1-m}O_{F}}.$

~

Since $s+1-m<s+1$, we conclude from the \eqref{eq:3.1-1},\eqref{eq:3.2}
that the norm map defines a bijection between the groups : 
\[
\widetilde{N}:\,\, H=\frac{O_{F}}{\varpi^{s+1-m}O_{F}}\rightarrow\frac{O_{E}}{\pi^{s+1-m}O_{E}}=H'.
\]

For $x\in O_{F}$ , we denote by $[x]_{H}$ the coset of $x$ in $H$
(and similarly for the rest of the group in this proof).

~

Since the integrand is well defined over $H$ , integrating over $O_{F}$
is equivalent to summing over $H$: 
\[
\underset{x\in O_{F}}{\int}\chi^{*}(1+\theta\pi^{m}N(x))dx=\mu_{F}(\varpi^{s+1-m}O_{F})\underset{[h]_{H}\in H}{\Sigma}\chi^{*}(1+\theta\pi^{m}N(h))
\]

~

We use the bijection $\widetilde{N}$ to calculate the sum over $H'$:

\[
\underset{[h]_{H}\in H}{\Sigma}\chi^{*}([1+\theta\pi^{m}N(h)]_{H})=\underset{[h']_{h'}\in H'}{\Sigma}\chi^{*}(1+\theta\pi^{m}h')
\]

~

The group $\frac{O_{E}}{\pi^{s+1-m}O_{E}}$ is isomorphic to the multiplicative
group:$U=\frac{1+\pi^{m}O_{F}}{1+\pi^{s+1}O_{F}}$ by the map $\tilde{\psi}:([x]_{H'})\mapsto[1+\pi^{m}\theta x]_{U}$,
so we get another bijection: 
\[
\tilde{\psi}:\frac{O_{F}}{\pi^{s+1-m}O_{F}}\rightarrow\frac{1+\pi^{m}O_{F}}{1+\pi^{s}O_{F}},
\]

The previous bijection ,$\tilde{\psi},$ tells us that summing over
$H'$ is equivalent to summing over the multiplicative group $U$
by $u=\tilde{\psi}(h),$ so

\[
\underset{[h']_{H'}\in H'}{\Sigma}\chi^{*}(1+\theta\pi^{m}h')=\underset{[u]_{U}\in U}{\Sigma}\chi^{*}(u).
\]

This sum surely vanish, since it is a summation of a non trivial character
of $U$.\end{proof}
\begin{lem}
\label{charchter lemma 2}(Case 3) Let $0<m<s$ , $\theta\in O_{E}^{*}$
then $\underset{x\in O_{F}^{*}}{\int}\chi^{*}(1+\theta\pi^{m}N(x))dx=0$\end{lem}
\begin{proof}
From Lemma \ref{charchter lemma 1}:
\[
\underset{x\in O_{F}}{\int}\chi^{*}(1+\theta\pi^{m}N(x))dx=0
\]

It is enough to prove that: 
\[
\underset{x\in\varpi O_{F}}{\int}\chi^{*}(1+\theta\pi^{m}N(x))dx=0.
\]

By substituting $x=\varpi y$, we get:

\[
\underset{x\in\varpi O_{F}}{\int}\chi^{*}(1+\theta\pi^{m}N(x))dx=q^{-1}\underset{y\in O_{F}}{\int}\chi^{*}(1+\theta\pi^{m+1}N(y))dy
\]

By Lemma \ref{charchter lemma 1} we conclude that the last integral
vanishes. 
\end{proof}
~
\begin{lem}
\label{Integral lemma} (Case 3) Let $F/E$ a ramified extension of
case 3 and $\eta\in O_{F}^{*}$ then if $Re(s_{1})>-\frac{1}{2}$,
and $s_{1}=-z_{1}+z_{2}-\frac{1}{2}$ we have
\begin{equation}
\underset{x\in O_{F}}{\int}|\eta+N(x)|^{s_{1}}dx=\frac{q^{2z_{2}}-q^{2z_{2}-1}+\chi^{*}(-\eta)\cdot(q^{s(2z_{1}-2z_{2})+2z_{1}}-q^{(2z_{1}-2z_{2})s+2z_{2}-1})}{q^{2z_{2}}-q^{2z_{1}}}
\end{equation}
\end{lem}
\begin{proof}
Suppose $-\eta\notin N(F)$ . Let $-\nu\in1+\pi^{s}O_{E}$ be an element
that is not a norm, note that $-\nu=(-\mu)\cdot\epsilon$ and $\epsilon\in N(F^{*})$,
that is $\epsilon=\alpha\overline{\alpha}$ for some $\alpha\in O_{F}^{*}$.
By making coordinate transformation $\alpha y=x$, we can assume w.l.o.g
that $-\eta\in1+\pi^{s}O_{E}$. 

~

In the coset $\varpi O_{F}$, it is easy to see that the integrand
is equal to $1$ .

~

The norm map induces an isomorphism $\widetilde{N_{1}}:\frac{O_{F}^{*}}{1+\varpi O_{F}}\rightarrow\frac{O_{E}^{*}}{1+\pi O_{E}}$
(The square map $x\to x^{2}$ of the finite group $\overline{F}^{*}=\overline{E}^{*}$)
and so $q-2$ cosets of $\frac{O_{F}^{*}}{1+\varpi O_{F}}$ do not
map by $\widetilde{N_{1}}$ to the coset $1+\pi O_{E}$ (the coset
of $-\eta$) and so $N(x)-(-\eta)\in O_{E}^{*}$ and the integrand
on those cosets is equal to $1$ .

~

So far we have:
\[
\underset{x\in O_{F}}{\int}|\eta+N(x)|^{s_{1}}dx=q^{-1}+\underset{x\in O_{F}^{*}}{\int}|\eta+N(x)|^{s_{1}}dx=
\]

\[
\underset{\mu_{F}(\varpi O_{F})}{\underbrace{q^{-1}}}+\underset{\mu_{F}[O_{F}^{*}-(1+\varpi O_{F})]}{\underbrace{(q-2)q^{-1}}}+\underset{x\in1+\varpi O_{F}}{\int}|\eta+N(x)|^{s_{1}}dx.
\]

~

Then for any $1\leq i<s$ we have that the norm map induce a isomorphism
\[
\widetilde{N_{2}}:\frac{1+\varpi^{i}O_{F}}{1+\varpi^{i+1}O_{F}}\rightarrow\frac{1+\pi^{i}O_{E}}{1+\pi^{i+1}OE}
\]
(The square map of the additive groups $\overline{F}=\overline{E}$).
So, for each $i$ , $q-1$ cosets of $\frac{1+\varpi^{i}O_{F}}{1+\varpi^{i+1}O_{F}}$
do not map to the coset $1+\pi^{i+1}O_{E}$ (the coset of $-\eta$)
and the value of the integrand is $q^{-2is_{1}}$. We have:

~
\[
\underset{x\in O_{F}}{\int}|\eta+N(x)|^{s_{1}}dx=q^{-1}+\frac{q-2}{q}+\underset{i=1}{\overset{s-1}{\sum}}q^{-2is_{1}}(q-1)q^{-i-1}+\underset{x\in1+\varpi^{s}O_{F}}{\int}|(-\eta)-N(x)|^{s_{1}}dx.
\]

~

Since $-\eta\notin N(1+\varpi^{s}O_{F})$ and $[(-\eta)+\pi^{s}O_{E}]\cap N(1+\varpi^{s}O_{F})=\phi$,
we have that in $1+\varpi^{s}O_{F}$ the integrand is equal to $q^{-2s_{1}\cdot s}$.

~

Altogether:

\[
\underset{x\in O_{F}}{\int}|\eta+N(x)|^{s_{1}}dx=(1-q^{-1})\underset{j=0}{\overset{s-1}{\sum}}q^{(-2s_{1}-1)j}+q^{s\cdot(-2s_{1}-1)}=
\]

\[
\frac{q^{2z_{2}}-q^{2z_{2}-1}+q^{(2z_{1}-2z_{2})s+2z_{2}-1}-q^{s(2z_{1}-2z_{2})+2z_{1}}}{q^{2z_{2}}-q^{2z_{1}}}.
\]

~

~

Suppose $-\eta\in N(F).$ Then $-\eta=N(\nu),\,\,\nu\in O_{F}^{*}$
by the substitution $\nu y=x$ and Lemma \eqref{measure lemma}: 
\[
\underset{x\in O_{F}}{\int}|\eta+N(x)|^{-z_{1}+z_{2}-\frac{1}{2}}dx=\underset{y\in O_{F}}{\int}|-1+N(y)|^{-z_{1}+z_{2}-\frac{1}{2}}dy=
\]
\[
(1-q^{-1})\overset{s-1}{\underset{j=0}{\sum}}q^{(2z_{1}-2z_{2})j}+2\overset{\infty}{\underset{j=s+1}{\sum}}(1-q^{-1})q^{(2z_{1}-2z_{2})j}+q^{s(2z_{1}-2z_{2})}
\]
\[
=\frac{q^{2z_{2}}-q^{2z_{2}-1}+q^{(2z_{1}-2z_{2})s+2z_{1}}-q^{s(2z_{1}-2z_{2})+2z_{2}-1}}{q^{2z_{2}}-q^{2z_{1}}}.
\]

\end{proof}

\section*{Representative of the K-Orbits (Ramified dyadic)}

\section{Classification of K-Orbits\label{sec:Classification-of-K-Orbits}}

~

We summarize results obtained by Jacobowitz \cite{key-10} that classified
Hermitian lattices over local fields ($O_{F}$-module equipped with
a Hermitian product $\langle\cdot,\cdot\rangle_{L}$ . An Hermitian
lattice with basis $\{x_{\lambda}\}$ can be represented by an Hermitian
matrix: $L_{\lambda\mu}=\langle x_{\lambda},x_{\mu}\rangle_{L}$ and
equivalence classes of lattices correspond to orbits of Hermitian
matrices under the action of $K$ given by $k\cdot x=kxk^{*}$ ,$k\in K,\, x\in X$.

~

The following definitions are also presented on \cite{key-10}. We
translate them in term of matrices: 

~
\begin{defn}
An hermitian matrix is called \textbf{\uline{$\varpi^{i}$ modular}}
if for every primitive vector $x\in M_{n,1}$ (that is, a vector $x=(x_{i}),\,\,\exists i_{0},1\leq i_{0}\leq n\,\, s.t\,\, x_{i_{0}}\in O_{F}^{*}$)
there is a vector $w\in M_{n,1}$ such that $w^{*}Ax=\varpi^{i}$. 
\end{defn}
~
\begin{defn}
\textbf{\uline{The norm ideal:}} $nL$ The ideal of $O_{F}$ generated
by elements $\langle v,v\rangle_{L},v\in M_{1,n}(O_{F})$ 
\end{defn}
~
\begin{defn}
\textbf{\uline{The scalar ideal }}$sL$ The Ideal generated by
$\langle v,w\rangle_{L},v,w\in M_{1,n}(O_{F})$
\end{defn}
~
\begin{defn}
$dL=det(L)\,(mod\, N(F)\,)\in E^{*}/N(F)$.
\end{defn}
~
\begin{defn}
$\varpi^{i}$-modular hyperbolic matrix: $H(i)\approx\left(\begin{array}{cc}
0 & \varpi^{i}\\
\overline{\varpi}^{i} & 0
\end{array}\right).$

~

The norm ideal, the scalar ideal, $dL$ and $\varpi^{i}-$modularity
are all invariants of the lattice under the action of $K$.
\end{defn}
Jacobowitz investigated the ramified non-dyadic case relevant for
our problem. 

~

~

It was shown :

~

\textbf{1.} Every Hermitian matrix is equivalent to the direct sum
of $\varpi^{i}$ modular $2\times2$ and $1\times1$ Hermitian matrices.

~

\textbf{2. }In the case of RP : $n\cdot H(i)=\varpi^{s+2i}O_{F}$

~~~~In the case of RU: $n\cdot H(i)=\varpi^{s-1+2i}O_{F}$

~

\textbf{3. }If $L$ is $\varpi^{i}$ modular then:$n\cdot H(i)\subseteq n\cdot L$

~

~

We conclude the following representatives of $K\backslash X$ in the
case :$n\cdot H(i)=n\cdot L:$

~

~

\textbf{RP}

~

1-modular matrices:

$1.\left(\begin{array}{cc}
0 & 1\\
1 & 0
\end{array}\right),\,\,\,2.\left(\begin{array}{cc}
\pi^{\frac{s}{2}} & 1\\
1 & -\pi^{\frac{s}{2}}\rho
\end{array}\right)$

~

$\varpi$-modular matrices:

~

$3.\left(\begin{array}{cc}
0 & \varpi\\
\overline{\varpi} & 0
\end{array}\right)$

\textbf{RU}

~

1-modular matrices:

$1.\left(\begin{array}{cc}
0 & 1\\
1 & 0
\end{array}\right)$

~

$\varpi$-modular matrices:

~

$2.\left(\begin{array}{cc}
\pi^{\frac{s+1}{2}} & \varpi\\
\overline{\varpi} & \pi^{\frac{s+1}{2}}\rho
\end{array}\right),\,\,3.\left(\begin{array}{cc}
0 & \varpi\\
\overline{\varpi} & 0
\end{array}\right)$

~

~

Now, suppose $n\cdot H(i)\subset n\cdot L=\varpi^{2m}O_{F}$ :

~

The other $\varpi^{i}-$modular planes: (see\cite{key-10} p. 459)

~

\textbf{~RP}

~

1-modular matrices:

$0<2m<s$

$4.\left(\begin{array}{cc}
\pi^{m} & 1\\
1 & 0
\end{array}\right),\,\,\,5.\left(\begin{array}{cc}
\pi^{m} & 1\\
1 & \pi^{s-m}\rho
\end{array}\right)$

~

~

$\varpi$-modular matrices:

$0<2m<s+2$

$6.\left(\begin{array}{cc}
\pi^{m} & \varpi\\
\overline{\varpi} & 0
\end{array}\right),\,\,\,7.\left(\begin{array}{cc}
\pi^{m} & \varpi\\
\overline{\varpi} & \pi^{s-m+1}\rho
\end{array}\right)$

\textbf{~}

\textbf{RU}

~

1-modular matrices:

$0<2m<s-1$

$4.\left(\begin{array}{cc}
\pi^{m} & 1\\
1 & 0
\end{array}\right),\,\,\,5.\left(\begin{array}{cc}
\pi^{m} & 1\\
1 & \pi^{s-m}\rho
\end{array}\right)$

~

$\varpi$-modular matrices:

$0<2m<s+1$

~$6.\left(\begin{array}{cc}
\pi^{m} & \varpi\\
\overline{\varpi} & 0
\end{array}\right),\,\,\,7.\left(\begin{array}{cc}
\pi^{m} & \varpi\\
\overline{\varpi} & \pi^{-m+s+1}\rho
\end{array}\right)$

~

~

We now deal with the case when the lattice is a sum of two $\pi^{i}$-modular
matrices:

~

We have the following representatives:

(Both RP and RU) 

~

$\left(\begin{array}{cc}
\pi^{\lambda_{1}}\epsilon_{1}\\
 & \pi^{\lambda_{2}}\epsilon_{2}
\end{array}\right)$, $\lambda_{1}\geq\lambda_{2},\,\,\epsilon_{i}\in\{1,\Delta\},\epsilon_{1}=1\,\, if\,\,\lambda_{1}-\lambda_{2}\leq s$

(see \cite{key-10} p. 463)

~

The representative we presented (beside the diagonals) are the 1 and
$\varpi$ modular representatives , up to multiplying by the scalar
$\pi^{a}\,\,\,\,(a\in\mathbb{Z}),$ this are all the representatives
of $K\backslash X$ .

~

We denote by $[K\backslash X]$ to be the set of representatives that
were have presented.

~

~

~

~

\section{The spherical functions for Case 3:}

~

$L(x,\chi_{1},\chi_{2},s_{1},s_{2})=\underset{}{\underset{K'}{\int}\chi_{1}(d_{1}(k\cdot x))\cdot|d_{1}(k\cdot x)|^{s_{1}}\chi_{2}(d_{2}(k\cdot x))\cdot|d_{2}(k\cdot x)|^{s_{2}}\, dk}$~

~

Substitute:

$s_{1}=-z_{1}+z_{2}-\frac{1}{2}$

$s_{2}=-z_{2}+\frac{1}{4}$

~

We denote : $z=(z_{1},z_{2}).$

~

For any $a_{1},a_{2},z_{1}z_{2}\in\mathbb{C}$ 
\[
\langle(a_{1},a_{2}),(b_{1},b_{2})\rangle=\langle(a_{1},a_{2}),z\rangle=a_{1}z_{1}+a_{2}z_{2}
\]

~

By abuse of notation we will some times denote:
\[
L(x,\chi_{1},\chi_{2},z_{1},z_{2})=L(x,\chi_{1},\chi_{2},s_{1}(z),s_{2}(z))
\]

~

From now on we will denote by $\sigma\in\{\sigma_{1},\sigma_{2}\}=\Sigma_{2}\subset Aut(C(q^{\pm z_{1}},q^{\pm z_{2}}))$
where :

\[
\sigma_{1}=Id
\]

\[
\sigma_{2}(q^{z_{1}})=q^{z_{2}},\,\,\,\sigma_{2}(q^{z_{2}})=q^{z_{1}}
\]
.

~
\begin{thm}
\label{main theorem}Let $x\in[K\backslash X]$ and 

\[
L(x,\chi_{1},\chi_{2},s_{1},s_{2})=\underset{}{\underset{K'}{\int}\underset{i=1}{\overset{2}{\prod}}d_{i}(\chi_{i}(k\cdot x))|d_{i}(k\cdot x)|^{s_{i}}dk}
\]

as was defined in Section \ref{sec:Introduction} then ~$L(x,\chi_{1},\chi_{2},z_{1},z_{2})$
is equal to the following :
\end{thm}
\textbf{\textit{\large RP: }}{\large \par}

~

\textbf{\uline{1.}} $L(x,\chi^{*},\chi_{2},z_{1},z_{2})=0$ unless
$x=\left(\begin{array}{cc}
\pi^{\lambda_{1}}\epsilon_{1}\\
 & \pi^{\lambda_{2}}\epsilon_{2}
\end{array}\right)$ with $\lambda_{1}-\lambda_{2}>s$ 

~

\textbf{\uline{2.}} If $x=\left(\begin{array}{cc}
\pi^{\lambda_{1}}\epsilon_{1}\\
 & \pi^{\lambda_{2}}\epsilon_{2}
\end{array}\right)$ with $\lambda_{1}-\lambda_{2}>s$ then: (Hironaka )

~~$L(x,\chi^{*},\chi_{2},z)=$

$\frac{q^{\frac{\lambda_{2}-\lambda_{1}}{2}}\chi^{*}(\epsilon_{2})\chi_{2}(\epsilon_{1}\epsilon_{2})}{1+q^{-1}}q^{2sz_{2}}(q^{2z_{2}}-q^{2z_{1}-1})\times\underset{\sigma\in\Sigma_{2}}{\sum}\sigma(\frac{q^{2\langle(\lambda_{1}-s,\lambda_{2}),z\rangle}}{q^{2z_{1}}-q^{2z_{2}}})$

~

~

\textbf{\uline{3.}}

\selectlanguage{american}%
$L(\left(\begin{array}{cc}
\pi^{\lambda_{1}}\epsilon_{1}\\
 & \pi^{\lambda_{2}}\epsilon_{2}
\end{array}\right),1,\chi_{2},z)=\frac{q^{\frac{\lambda_{2}-\lambda_{1}}{2}}\chi_{2}(\epsilon_{1}\epsilon_{2})q^{-2sz_{2}}}{(q^{-1}+1)(q^{2z_{2}}-q^{2z_{1}})}\times$

$[\chi^{*}(-\epsilon_{1}\epsilon_{2})q^{<(\lambda_{1}+s,\lambda_{2}),2z>}(q^{2z_{1}}-q^{2z_{2}-1})+q^{\langle(\lambda_{2},\lambda_{1}+s),2z\rangle}(q^{2z_{2}}-q^{2z_{1}-1})]$

~

\selectlanguage{english}%
~

\textbf{\uline{4.}}

$L(\left(\begin{array}{cc}
 & 1\\
1
\end{array}\right),1,\chi_{2},z)=\frac{\chi_{2}(-1)(1-q^{-1})q^{l-2s\cdot z_{2}}}{q^{-1}+1}[\frac{q^{s(z_{1}+z_{2})}(q^{2z_{1}}+q^{2z_{2}})}{q^{2z_{2}}-q^{2z_{1}}}]$

~

~

\textbf{\uline{5.}}

$L(\left(\begin{array}{cc}
 & \varpi\\
\overline{\varpi}
\end{array}\right),1,\chi_{2},z)==\chi_{2}(-1)(1-q^{-1})q^{l+\frac{1}{2}}q^{-2sz_{2}}q^{s(z_{1}+z_{2})}[\frac{q^{2z_{1}+2z_{2}}}{q^{2z_{2}}-q^{2z_{1}}}]$

~

~

\textbf{\uline{6.}}

$L(\left(\begin{array}{cc}
\pi^{\frac{s}{2}} & 1\\
1 & -\pi^{\frac{s}{2}}\rho
\end{array}\right),1,\chi_{2},z)=\chi_{2}(-\Delta)q^{-2s\cdot z_{2}+l}q^{s\cdot(z_{1}+z_{2})}$

~

~

\textbf{\uline{7.}}

$0<m<\frac{s}{2}$

$L(\left(\begin{array}{cc}
\pi^{m} & 1\\
1 & 0
\end{array}\right),1,\chi_{2},z)=$

$=\frac{\chi_{2}(-1)}{q^{-1}+1}\frac{q^{m-2sz_{2}}(1-q^{-1})}{q^{2z_{2}}-q^{2z_{1}}}\underset{\sigma}{\sum}\sigma(q^{\langle(m,s-m+1),2z\rangle})$

~

~

\textbf{\uline{~8.}}

$\,0<m<\frac{s}{2}+1$

$L(\left(\begin{array}{cc}
\pi^{m} & \varpi\\
\overline{\varpi} & 0
\end{array}\right),1,\chi_{2},z)=$

$=\frac{\chi_{2}(-1)q^{-\frac{1}{2}}}{q^{-1}+1}\frac{q^{m-2sz_{2}}}{q^{2z_{2}}-q^{2z_{1}}}[\underset{\sigma}{\sum}\sigma(q^{\langle(m,2+s-m),2z\rangle}-q^{\langle(s+1-m,m+1),2z\rangle-1})]$

~

~

\textbf{\uline{9.}}

$0<m<\frac{s}{2}$

$L(\left(\begin{array}{cc}
\pi^{m} & 1\\
1 & -\pi^{s-m}\rho
\end{array}\right),1,\chi_{2},z)=$

$\frac{\chi_{2}(-\Delta)q^{m-2sz_{2}}}{q^{-1}+1}[\underset{\sigma}{\sum}\sigma(\frac{q^{\langle(m,1-m+s),2z\rangle}}{q^{2z_{2}}-q^{2z_{1}}})]$

~

~

~

\textbf{\uline{10.}}

$0<m<\frac{s}{2}+1$

$L(\left(\begin{array}{cc}
\pi^{m} & \varpi\\
\overline{\varpi} & -\pi^{s-m}\rho
\end{array}\right),1,\chi_{2},z)=$

$=\frac{\chi_{2}(-\Delta)q^{m-\frac{1}{2}-s2z_{2}}}{q^{-1}+1}[\underset{\sigma}{\sum}\sigma(\frac{q^{\langle(m,2+s-m),2z\rangle}}{q^{2z_{2}}-q^{2z_{1}}})]$

~

~

\textbf{\textit{\large RU}}{\large \par}

~~~~~~

~1,2,3 are the same as the RP case.

~

\textbf{\uline{4.}}

$L(\left(\begin{array}{cc}
 & 1\\
1
\end{array}\right),1,\chi_{2},z)=q\chi_{2}(-1)(1-q^{-1})q^{-2sz_{2}}[\frac{q^{(s+1)(z_{1}+z_{2})}}{q^{2z_{2}}-q^{2z_{1}}}]$

~

\textbf{\uline{5.}}

$L(\left(\begin{array}{cc}
 & \varpi\\
\overline{\varpi}
\end{array}\right),1,\chi_{2},z)=\frac{\chi_{2}(-1)(1-q^{-1})q^{l+\frac{1}{2}}q^{-2s\cdot z_{2}}}{1+q^{-1}}q^{(s+1)(z_{1}+z_{2})}[\frac{q^{2z_{1}}+q^{2z_{2}}}{q^{2z_{2}}-q^{2z_{1}}}]$

~

\textbf{\uline{6.}}

~

~$L(x,\left(\begin{array}{cc}
\pi^{\frac{s+1}{2}} & \varpi\\
\overline{\varpi} & -\pi^{\frac{s+1}{2}}\rho
\end{array}\right),\chi_{2},z)=\chi_{2}(-\Delta)q^{-2s\cdot z_{2}+\frac{s}{2}}[q^{(s+1)(z_{1}+z_{2})}]$

~

For the following representatives:

\textbf{\uline{7.}}$\left(\begin{array}{cc}
\pi^{m} & 1\\
1 & 0
\end{array}\right)$,~\textbf{\uline{8.}}~$\left(\begin{array}{cc}
\pi^{m} & \varpi\\
\overline{\varpi} & 0
\end{array}\right)$,~\textbf{\uline{9.}}$\left(\begin{array}{cc}
\pi^{m} & 1\\
1 & -\pi^{s-m}\rho
\end{array}\right)$ \textbf{\uline{10.}} $\left(\begin{array}{cc}
\pi^{m} & \varpi\\
\overline{\varpi} & -\pi^{s-m}\rho
\end{array}\right)$

the result is the same as RP Case..

~

And any $a\in E^{*}$,~~
\begin{equation}
L(ax,\chi,s_{1},s_{2})=|a|^{s_{1}+2s_{2}}\chi_{1}(a)\cdot\chi_{2}(a)^{2}L(x,\chi_{1},\chi_{2},s_{1},s_{2})\label{equation 6.1}
\end{equation}

~

~

So one can calculate the spherical function for any $K$ orbit in
$X$.

~

~

~

The proof of the theorem will be given in chapters 7 and 8.

~

~

\subsection{Functional Equations}

~

A corollary that follows from Theorem \eqref{main theorem} is the
following functional equations:

~

Define $\tilde{L}(x,\chi_{1},\chi_{2},z)=q^{2sz_{2}}L(x,\chi_{1},\chi_{2},z)$

~

Then :

~

$\tilde{L}(x,1,\chi_{2},z_{2},z_{1})=-\chi^{*}(-1)\tilde{L}(x,1,\chi^{*}\chi_{2},z_{1},z_{2})$

~

~

$\tilde{L}(x,\chi^{*},\chi_{2},z_{2},z_{1})=q^{4s(z_{1}-z_{2})}\frac{q^{2z_{1}}-q^{2z_{2}-1}}{q^{2z_{2}}-q^{2z_{1}-1}}\tilde{L}(x,\chi^{*},\chi_{2},z_{1},z_{2})$

~

And actually if we define:

~

$\chi_{1}=\omega_{1}^{-1}\omega_{2}$

$\chi_{2}=\omega_{2}^{-1}$

~

~

We have the following functional equations :

~

If $\omega_{1}\omega_{2}^{-1}=1$ then :

~

$\tilde{L}(x,\omega_{2},\omega_{1},z_{2},z_{1})=-\chi^{*}(-1)\tilde{L}(x,\chi^{*}\omega_{1},\chi^{*}\omega_{2},z_{1},z_{2})$ 

~

If $\omega_{1}\omega_{2}^{-1}\neq1$ then:

$\tilde{L}(x,\omega_{2},\omega_{1},z_{2},z_{1})=q^{4s(z_{1}-z_{2})}\frac{q^{2z_{1}}-q^{2z_{2}-1}}{q^{2z_{2}}-q^{2z_{1}-1}}\tilde{L}(x,\chi^{*}\omega_{1},\chi^{*}\omega_{2},z_{1},z_{2})$

~
\begin{rem}
The previous transformation of characters comes form the fact that
if we define $\theta_{1}(a)=w_{1}(a)|a|^{z_{1}}$ and $\theta_{2}(a)=w_{2}(a)|a|^{z_{2}}$
then for $p=\left(\begin{array}{cc}
a & x\\
 & d
\end{array}\right)$ we have $d(p\cdot x)=\chi_{1}(d_{1}(p\cdot x))\cdot|d_{1}(p\cdot x)|^{s_{1}}\cdot\chi_{2}(d_{2}(p\cdot x))\cdot|d_{2}(p\cdot x)|^{s_{2}}=\theta_{1}(a)\cdot\theta_{2}(d)\cdot\delta_{P}^{\frac{1}{2}}(p)\cdot d(x)$
, where $\delta_{P}(p)$ is the modular character of the Borel subgroup,
this property is related to the principal series representation of
$GL_{2}(F)$. 
\end{rem}

\section{Calculation of $L(x,1,\chi_{2},z)$ :}

~

We will compute the integral over cosets of the Iwahori subgroup,
$B$

$B=\left(\begin{array}{cc}
a & b\\
c & d
\end{array}\right)\in K,\,\, b\in\varpi O_{F}$.

~

$B$ has the factorization:

~

$B=N_{-}\,\, A\,\, N^{+}(\varpi O_{F})=\left(\begin{array}{cc}
1 & 0\\
t & 1
\end{array}\right)\left(\begin{array}{cc}
a_{1} & 0\\
0 & a_{2}
\end{array}\right)\left(\begin{array}{cc}
1 & y\\
0 & 1
\end{array}\right)=\left(\begin{array}{cc}
a_{1} & a_{1}y\\
a_{1}t\,\,\, & a_{1}ty+a_{2}
\end{array}\right)$,

~

Where $a_{1},\, a_{2}\in O_{F}^{*},\,\, t\in O_{F},\,\, y\in\varpi O_{F}$.

~

The Haar measure on $B$ , $\mu_{B}$ is taken to be $dt\times da_{1}\times da_{2}\times dy$
, where:

~

$dt-$the Haar measure on $O_{F}$ normalized to be 1, $da_{1},da_{2}$-
the Haar measure of $O_{F}^{*}$ normalized to be 1 and $dy$- the
Haar measure on $\varpi O_{F}$ normalized to be $q^{-1}$.

~

Overall we have $\mu_{B}(B)=q^{-1}$.

~

Here is a list of coset representatives for $B\backslash K$ :

~$b_{i}=\left(\begin{array}{cc}
1 & r_{i}\\
0 & 1
\end{array}\right)$ if $0\leq i\leq q$ ~~and $r_{i}$ are representative of $\bar{F}$,
$q=|\bar{F}|$ 

and ~~$b_{q+1}=\left(\begin{array}{cc}
 & 1\\
1
\end{array}\right).$

~

So:

$Bb_{i}=\left(\begin{array}{cc}
a_{1} & a_{1}y+a_{1}r_{i}\\
a_{1}t\,\,\, & a_{1}t(y+r_{i})+a_{2}
\end{array}\right)$~if $1\leq i\leq q$ and $Bb_{q+1}=\left(\begin{array}{cc}
a_{1}y & a_{1}\\
a_{1}ty+a_{2}\,\, & a_{1}t
\end{array}\right)$.

~

Instead of integrating on $K$, we integrated over the different cosets
of $B$ using the Haar measure of $B$.

~

We want to normalize the Haar measure to be $\mu_{G}(K)=1$ ,so we
normalize the measure on each coset by multiplying by $\frac{q}{q+1}$.

~

Let $A=\left(\begin{array}{cc}
a & c\\
\bar{c} & b
\end{array}\right)$ be a fixed Hermitian matrix and $C=k\cdot A\cdot\overline{k^{t}}$
then using the factorization on $k$ one can parametrize $C$ as follows: 

\begin{equation}
d_{1}(C)=\begin{cases}
\begin{array}{cc}
N(a_{1})\cdot[a+Tr(\bar{c}(y+r_{i}))+bN(y+r_{i})]\,\,\,\,\, & k\in Bb_{i}\,\, i<q+1\\
N(a_{1})\cdot[aN(y)+Tr(\bar{c}y)+b] & k\in Bb_{q+1}
\end{array}\end{cases}\label{eq:7.1}
\end{equation}

~

In the first step of the proof we show that we can actually integrate
over $K$ instead of $K'$, from now on we will assume that the characters
of $F^{*}$are defined on $0$, say:$\chi_{i}(0)=0.$ 
\begin{lem}
Let $A\in X$ be a fixed Hermitian matrix. Then $\mu_{K}$ measure
of the set $\{k\in K|\,\, d_{1}(k\cdot A)=0\}$ is $0$ . \end{lem}
\begin{proof}
By \eqref{eq:7.1}, it is sufficient to show that the inverse image
of a point of the trace and norm maps is of measure zero. Note that
the trace and norm maps: $Tr:O_{F}\to Tr(O_{F})$,~~$N:\, O_{F}^{*}\to N(O_{F}^{*})$
are surjective homomorphisms, and that the Haar measure of $Tr(O_{F}),N(O_{F}^{*})$
(being an open sets) is the induced measure of $O_{E},O_{E}^{*}$(respectively).
In particular the measure of a singleton is $0$. We use Lemma \ref{pushforword lemma}
to deduce that the measure of the inverse images is also $0$. 
\end{proof}
~

~
\begin{lem}
\label{formula lemma}

\[
L(x,\chi_{1},\chi_{2},s_{1},s_{2})=\underset{}{\underset{K}{\int}\chi_{1}(d_{1}(k\cdot x))\cdot|d_{1}(k\cdot x)|^{s_{1}}\chi_{2}(d_{2}(k\cdot x))\cdot|d_{2}(k\cdot x)|^{s_{2}}\, dk=}
\]

\[
\frac{q\chi_{2}(det(x))|det(x)|^{s_{2}}}{q+1}\underset{O_{F}}{[\int}|a+Tr(\bar{c}t)+bN(t)|^{s_{1}}\chi_{1}(a+Tr(\bar{c}t)+bN(t))dt+
\]

\end{lem}
\[
\underset{\varpi O_{F}}{\int}|aN(y)+Tr(\overline{c}y)+b|^{s_{1}}\chi_{1}(aN(y)+Tr(\overline{c}y)+b)dy].
\]

~

~
\begin{proof}
For any $f\in L^{1}(X)$: 
\[
\underset{K}{\int}f(k\cdot x)\, dk=\frac{q}{q+1}\underset{i=1}{\overset{q+1}{\sum}}\underset{k=bb_{i}\in Bb{}_{i}}{\int}f(k\cdot x)\, db
\]

Using \eqref{eq:7.1} and the Haar measure of $B$, We have:{\small 
\begin{eqnarray*}
L(x,\chi_{1},\chi_{2},s_{1},s_{2})= & \frac{q\chi_{2}(det(x))|det(x)|^{s_{2}}}{q+1}[\underset{i=1}{\overset{q}{\sum}}\underset{y\in\varpi O_{F}}{\int}\chi_{1}(a+Tr(\bar{c}(y+r_{i}))+bN(y+r_{i}))\cdot|a+Tr(\bar{c}(y+r_{i}))+bN(y+r_{i})|^{s_{1}}dy
\end{eqnarray*}
}{\small \par}

{\small 
\[
+\underset{y\in\varpi O_{F}}{\int}\chi_{1}(aN(y)+Tr(\bar{c}y)+b)\cdot|aN(y)+Tr(\bar{c}y)+b|^{s_{1}}dy].
\]
}{\small \par}

The last equation follows from the facts that $|N(a_{1})|=1$, $\chi_{1}(N(a_{1}))=1$
and that $|det(k\cdot x)|=|det(x)|$.

Notice that performing the first sum is equivalent to integrating
over $O_{F}$ . Altogether we obtain our formula .
\end{proof}

\subsection{Calculation $L(x,1,\chi_{2},z)$ on the diagonal elements}

~

We use Lemma \ref{formula lemma} to calculate the various cases.

~

On the diagonal representatives:

~

$x\sim\left(\begin{array}{cc}
\pi^{\lambda_{1}}\epsilon_{1}\\
 & \pi^{\lambda_{2}}\epsilon_{2}
\end{array}\right)$

~

~

~

\[
L(x,1,\chi_{2},s_{1},s_{2})=
\]

\[
\frac{q^{(\lambda_{1}+\lambda_{2})-2s_{2}}\chi_{2}(\epsilon_{1}\epsilon_{2})}{q^{-1}+1}[\underset{O_{F}}{\int}|\pi^{\lambda_{1}}\epsilon_{1}+\pi^{\lambda_{2}}\epsilon_{2}N(t)|^{s_{1}}dt+\underset{\varpi O_{F}}{\int}|\epsilon_{2}\pi^{\lambda_{2}}+\epsilon_{1}\pi^{\lambda_{1}}N(y)|^{s_{1}}dy]=
\]

~

\selectlanguage{american}%
\[
\underset{}{\frac{q^{(\lambda_{1}+\lambda_{2})-2s_{2}}\chi_{2}(\epsilon_{1}\epsilon_{2})\cdot q^{-\lambda_{2}2s_{1}}}{q^{-1}+1}[\underset{O_{F}}{\int}}|\pi^{\lambda_{1}-\lambda_{2}}\frac{\epsilon_{1}}{\epsilon_{2}}+N(t)|^{s_{1}}dt+\underset{\varpi O_{F}}{\int}|1+\frac{\epsilon_{1}}{\epsilon_{2}}\pi^{\lambda_{1}-\lambda_{2}}N(y)|^{s_{1}}dy].
\]

\selectlanguage{english}%
~

The second integrand is the constant $1$ since $|N(y)|<1$ if $y\in\varpi O_{F}$.

Note that for any $t\in O_{F}-\varpi^{\lambda_{1}-\lambda_{2}}O_{F}$,
we have $|\pi^{\lambda_{1}-\lambda_{2}}\frac{\epsilon_{1}}{\epsilon_{2}}+N(t)|^{s_{1}}=|t|^{2s_{1}}$
and so we can calculate this integral easily on this space:

\selectlanguage{american}%
~

\selectlanguage{english}%
~

\selectlanguage{american}%
$\underset{}{L(x,1,\chi_{2},s_{1},s_{2})=\frac{q^{(\lambda_{1}+\lambda_{2})-2s_{2}}\cdot q^{-\lambda_{2}2s_{1}}\chi_{2}(\epsilon_{1}\epsilon_{2})}{q^{-1}+1}[\underset{}{\overset{\lambda_{1}-\lambda_{2}-1}{\underset{j=0}{\sum}}q^{(-2s_{1}-1)j}(1-q^{-1})+}}$

~

\[
q^{(\lambda_{1}-\lambda_{2})(-2s_{1})}\underset{\varpi^{\lambda_{1}-\lambda_{2}}O_{F}}{\int}|\frac{\epsilon_{1}}{\epsilon_{2}}+N(t)|^{s_{1}}dt+q^{-1}].
\]

~

\selectlanguage{english}%
After the substitution $\varpi^{\lambda_{1}-\lambda_{2}}y=t$, we
get:

~

\selectlanguage{american}%
\[
\underset{}{L(x,1,\chi_{2},s_{1},s_{2})=\frac{q^{(\lambda_{1}+\lambda_{2})-2s_{2}}\cdot q^{-\lambda_{2}2s_{1}}\chi_{2}(\epsilon_{1}\epsilon_{2})}{q^{-1}+1}[(1-q^{-1})\frac{1-q^{(-2s_{1}-1)(\lambda_{1}-\lambda_{2})}}{1-q^{(-2s_{1}-1)}}+}
\]

\[
+q^{(\lambda_{1}-\lambda_{2})(-2s_{1}-1)}\underset{O_{F}}{\int}|\frac{\epsilon_{1}}{\epsilon_{2}}+N(y)|^{s_{1}}dy+\frac{1}{q}].
\]

~

\[
L(x,1,\chi_{2},z)=\frac{q^{\frac{\lambda_{2}-\lambda_{1}}{2}}\cdot q^{2\lambda_{1}z_{2}+2\lambda_{2}z_{1}}\chi_{2}(\epsilon_{1}\epsilon_{2})}{q^{-1}+1}[\frac{q^{2z_{2}}-q^{2z_{2}-1}-q^{(2z_{1}-2z_{2})(\lambda_{1}-\lambda_{2})+2z_{2}}+q^{(2z_{1}-2z_{2})(\lambda_{1}-\lambda_{2})+2z_{2}-1}}{q^{2z_{2}}-q^{2z_{1}}}
\]

~

\begin{equation}
+q^{(\lambda_{1}-\lambda_{2})(2z_{1}-2z_{2})}\underset{O_{F}}{\int}|\frac{\epsilon_{1}}{\epsilon_{2}}+N(t)|^{s_{1}}dt+q^{-1}].\label{eq:7.2}
\end{equation}

~

\selectlanguage{english}%
There are two cases:

~

\textbf{\emph{\large Case 1:}} $-\frac{\epsilon_{1}}{\epsilon_{2}}\notin N(F):$
We have by Lemma \ref{Integral lemma} : 
\[
\underset{O_{F}}{\int}|\frac{\epsilon_{1}}{\epsilon_{2}}+N(t)|dt=\frac{q^{2z_{2}}-q^{2z_{2}-1}+q^{(2z_{1}-2z_{2})s+2z_{2}-1}-q^{s(2z_{1}-2z_{2})+2z_{1}}}{q^{2z_{2}}-q^{2z_{1}}}.
\]
 We set this into Eq \ref{eq:7.2}:\foreignlanguage{american}{
\[
L(x,1,\chi_{2},z)=\underset{}{\frac{q^{\frac{\lambda_{2}-\lambda_{1}}{2}}\chi_{2}(\epsilon_{1}\epsilon_{2})q^{-2sz_{2}}}{q^{-1}+1}[\frac{q^{2z_{2}+2(\lambda_{1}+s)z_{2}+2\lambda_{2}z_{1}}}{q^{2z_{2}}-q^{2z_{1}}}}+
\]
}

\selectlanguage{american}%
~

\[
+\frac{q^{2z_{1}(s+\lambda_{1})+2z_{2}\lambda_{2}+2z_{2}-1}-q^{2z_{1}(s+\lambda_{1})+2z_{2}\lambda_{2}+2z_{1}}}{q^{2z_{2}}-q^{2z_{1}}}+\frac{-q^{2z_{1}-1+2(\lambda_{1}+s)z_{2}+2\lambda_{2}z_{1}}}{q^{2z_{2}}-q^{2z_{1}}}=
\]

~

\[
\frac{q^{\frac{\lambda_{2}-\lambda_{1}}{2}}\chi_{2}(\epsilon_{1}\epsilon_{2})q^{-2sz_{2}}}{q^{-1}+1}[q^{\langle(s+\lambda_{1},\lambda_{2}),2z\rangle}(q^{2z_{2}-1}-q^{2z_{1}})+q^{\langle(\lambda_{2},\lambda_{1}+s),2z\rangle}(q^{2z_{2}}-q^{2z_{1}-1}).
\]

~

\selectlanguage{english}%
\textbf{\emph{\large Case 2:}}$-\frac{\epsilon_{1}}{\epsilon_{2}}\in N(F)$:
We have 
\[
\underset{O_{F}}{\int}|\frac{\epsilon_{1}}{\epsilon_{2}}+N(t)|^{s_{1}}dt=\frac{q^{2z_{2}}-q^{2z_{2}-1}-q^{s(2z_{1}-2z_{2})+2z_{2}-1}+q^{(2z_{1}-2z_{2})s+2z_{1}}}{q^{2z_{2}}-q^{2z_{1}}}.
\]

Setting this into Eq \ref{eq:7.2}:

\[
\underset{}{L(x,1,\chi_{2},z)=\underset{}{\frac{q^{\frac{\lambda_{2}-\lambda_{1}}{2}}\chi_{2}(\epsilon_{1}\epsilon_{2})}{q^{-1}+1}[\frac{q^{2z_{2}+2(\lambda_{1}+s)z_{2}+2\lambda_{2}z_{1}}}{q^{2z_{2}}-q^{2z_{1}}}}+}
\]

~

\[
\frac{+q^{2z_{1}(s+\lambda_{1})+(\lambda_{2}-s)2z_{2}+2z_{1}}-q^{2z_{1}(s+\lambda_{1})+(\lambda_{2}-s)2z_{2}+2z_{2}-1}}{q^{2z_{2}}-q^{2z_{1}}}+\frac{-q^{2z_{1}-1+2(\lambda_{1}+s)z_{2}+2\lambda_{2}z_{1}}}{q^{2z_{2}}-q^{2z_{1}}}]=
\]

~

\[
\frac{q^{\frac{\lambda_{2}-\lambda_{1}}{2}}\chi_{2}(\epsilon_{1}\epsilon_{2})q^{-2sz_{2}}}{(q^{-1}+1)(q^{2z_{2}}-q^{2z_{1}})}[q^{\langle(\lambda_{1}+s,\lambda_{2}),2z\rangle}(q^{2z_{1}}-q^{2z_{2}-1})+q^{\langle(\lambda_{2},\lambda_{1}+s),2z\rangle}(q^{2z_{2}}-q^{2z_{1}-1})].
\]

~

\subsection{Calculating $L(x,1,\chi_{2},z)$ on the non diagonal elements:}

~

Let $x=\left(\begin{array}{cc}
a & \varpi^{c}\\
\overline{\varpi}^{c} & b
\end{array}\right)$. Using Lemma \ref{formula lemma} :

~

~

\[
L(x,1,\chi_{2},s_{1},s_{2})=\frac{q\chi_{2}(det(x))|det(x)|^{s_{2}}}{q+1}\underset{O_{F}}{[\int}|a+Tr(\overline{\varpi}^{c}t)+bN(t)|^{s_{1}}dt+
\]

\[
\underset{\varpi O_{F}}{\int}|aN(y)+Tr(\overline{\varpi}^{c}y)+b|^{s_{1}}dy].
\]

~

~

In the following section we will make use of Lemma \ref{pushforword lemma}
and \ref{sub:The-Trace} (coordinate substitution).

~

\subsubsection{RP case}

~

Recall :~~$l=\frac{s}{2}$,~ $Tr(\varpi^{2i}O_{F})=\pi^{l+i}O_{E},\,\,\,$$Tr(\varpi^{2i-1}O_{F})=\pi^{l+i}O_{E}.$

~

\textbf{\uline{1.}}

$x\sim\left(\begin{array}{cc}
 & 1\\
1
\end{array}\right)$

~

\[
L(x,1,\chi_{2},s_{1},s_{2})=\frac{\chi_{2}(-1)}{q^{-1}+1}[\underset{x\in O_{F}}{\int}|Tr(t)|^{s_{1}}dt+\underset{}{\underset{y\in\varpi O_{F}}{\int}|Tr(y)|^{s_{1}}dy}].
\]

~

We substitute $u=Tr(t)$,~~$v=Tr(y)$ 

~

\[
L(x,1,\chi_{2},s_{1},s_{2})=\frac{q^{l}\chi_{2}(-1)}{q^{-1}+1}[\underset{u\in\pi^{l}O_{E}}{\int}|u|^{s_{1}}du+\underset{}{\underset{v\in\pi^{l+1}O_{E}}{\int}|v|^{s_{1}}dv}].
\]

~

The integrals above are simple and reduces to geometric sums:

~

~

\[
L(x,1,\chi_{2},s_{1},s_{2})=\frac{\cdot\chi_{2}(-1)}{q^{-1}+1}[\overset{\infty}{\underset{i=0}{\sum}}q^{-i}(1-q^{-1})q^{(l+i)(-2s_{1})}+\overset{\infty}{\underset{i=1}{\sum}}q^{-i}(1-q^{-1})q^{(l+i)(-2s_{1})}]=
\]

~

\[
\frac{\chi_{2}(-1)(1-q^{-1})q^{-2ls_{1}}}{q^{-1}+1}[\overset{\infty}{\underset{i=0}{\sum}}q^{(-2s_{1}-1)i}+\overset{\infty}{\underset{i=1}{\sum}}q^{(-2s_{1}-1)i}]=
\]

~

\[
\frac{\chi_{2}(-1)(1-q^{-1})q^{-2ls_{1}}}{q^{-1}+1}[\frac{1}{1-q^{-2s_{1}-1}}+\frac{q^{-2s_{1}-1}}{1-q^{-2s_{1}-1}}].
\]

~

After simplifying this expression and performing the transformation
$s\mapsto z$, we obtain:

~

\[
L(x,1,\chi_{2},z)=\frac{\chi_{2}(-1)(1-q^{-1})q^{l-4lz_{2}}}{q^{-1}+1}[\frac{q^{s(z_{1}+z_{2})}(q^{2z_{1}}+q^{2z_{2}})}{q^{2z_{2}}-q^{2z_{1}}}].
\]

~

~

\textbf{\uline{2.}}

$x\sim\left(\begin{array}{cc}
 & \varpi\\
\overline{\varpi}
\end{array}\right)$

~

\[
L(x,1,\chi_{2},s_{1},s_{2})=\frac{\chi_{2}(-1)q^{-2s_{2}}}{q^{-1}+1}[\underset{x\in O_{F}}{\int}|Tr(\overline{\varpi}t)|^{s_{1}}dt+\underset{}{\underset{y\in\varpi O_{F}}{\int}|Tr(\overline{\varpi}y)|^{s_{1}}dy}].
\]

~

\[
\frac{\chi_{2}(-1)q^{-2s_{2}+1}}{q^{-1}+1}[\underset{u\in\varpi O_{F}}{\int}|Tr(u)|^{s_{1}}du+\underset{}{\underset{u\in\varpi^{2}O_{F}}{\int}|Tr(u)|^{s_{1}}du}].
\]

~

Similarly to the previous calculations, we substitute $v=Tr(t),\, y=Tr(u)$,
and use Lemma \ref{pushforword lemma}. Note that the image of both
transformations coincides $Im(\varpi O_{F})=Im(\varpi^{2}O_{F})=\pi^{l+1}O_{E}$,
but the multiplicative factors (The {}``Jacobians'') are different:

~

\[
L(x,1,\chi_{2},s_{1},s_{2})=\frac{\chi_{2}(-1)q^{-2s_{2}+1}}{q^{-1}+1}[q^{l}\underset{u\in\pi^{l+1}O_{E}}{\int}|v|^{s_{1}}dv+\underset{}{q^{l-1}\underset{u\in\pi^{l+1}O_{E}}{\int}|y|^{s_{1}}dy}]=
\]

~

\[
\frac{\chi_{2}(-1)q^{-2s_{2}+1}q^{l}}{q^{-1}+1}[\overset{\infty}{\underset{j=1}{\sum}}q^{-l-i}(1-q^{-1})q^{(l+i)(-2s_{1})}+q^{-1}\overset{\infty}{\underset{j=1}{\sum}}q^{-l-i}(1-q^{-1})q^{(l+i)(-2s_{1})}].
\]

~

Substituting $s\mapsto z$:

~

\[
L(x,1,\chi_{2},z)=\frac{\chi_{2}(-1)(1-q^{-1})q^{2z_{2}+\frac{1}{2}}}{q^{-1}+1}q^{l(2z_{1}-2z_{2}+1)}[\overset{\infty}{\underset{j=1}{\sum}}q^{i(2z_{1}-2z_{2})}+q^{-1}\overset{\infty}{\underset{j=1}{\sum}}q^{i(2z_{1}-2z_{2})}]=
\]

~

\[
\frac{\chi_{2}(-1)(1-q^{-1})q^{2z_{2}+\frac{1}{2}}}{q^{-1}+1}q^{l(2z_{1}-2z_{2}+1)}[\frac{q^{2z_{1}}}{q^{2z_{2}}-q^{2z_{1}}}+\frac{q^{2z_{1}-1}}{q^{2z_{2}}-q^{2z_{1}}}]=
\]

~

\[
\frac{\chi_{2}(-1)(1-q^{-1})q^{\frac{1}{2}}}{q^{-1}+1}q^{-2sz_{2}}q^{l(2z_{1}+2z_{2}+1)}[\frac{q^{2z_{2}+2z_{1}}}{q^{2z_{2}}-q^{2z_{1}}}+\frac{q^{2z_{2}+2z_{1}-1}}{q^{2z_{2}}-q^{2z_{1}}}].
\]

~

After simplifying this expression, we get:

~

\[
L(x,1,\chi_{2},z)=\chi_{2}(-1)(1-q^{-1})q^{l+\frac{1}{2}}q^{-2sz_{2}}q^{s(z_{1}+z_{2})}[\frac{q^{2z_{1}+2z_{2}}}{q^{2z_{2}}-q^{2z_{1}}}].
\]

~

~

\textbf{\uline{3.}}

$x\sim\left(\begin{array}{cc}
\pi^{\frac{s}{2}} & 1\\
1 & -\pi^{\frac{s}{2}}\rho
\end{array}\right),\,\,\,\,1+\pi^{s}\rho=\Delta\notin N(F^{*})$

~

\[
L(x,1,\chi_{2},s_{1},s_{2})=\frac{\chi_{2}(-\Delta)}{q^{-1}+1}\underset{O_{F}}{[\int}|\pi^{\frac{s}{2}}+Tr(t)-\pi^{\frac{s}{2}}\rho N(x)|^{s_{1}}dt+\underset{\varpi O}{\int}|\pi^{\frac{s}{2}}N(y)+Tr(y)-\pi^{\frac{s}{2}}\rho|^{s_{1}}dy].
\]

~

Because $|Tr(y)|,|\pi^{\frac{s}{2}}N(y)|<1$ ~(see Sections \ref{sub:The-Trace}
and \ref{sub:The-Norm}), we conclude that the second integrand is
constant in $\varpi O_{F}$.

~

\begin{equation}
L(x,1,\chi_{2},s_{1},s_{2})=\frac{\chi_{2}(-\Delta)}{q^{-1}+1}\underset{O_{F}^{*}}{[\int}|\pi^{\frac{s}{2}}+Tr(t)-\pi^{\frac{s}{2}}\rho N(t)|^{s_{1}}dt+q^{-s\cdot s_{1}-1}+q^{-s\cdot s_{1}-1}].\label{eq:7.3}
\end{equation}

~

Since $|N(t)|=1$ if $t\in O_{F}^{*}$ , and $N(t^{-1})Tr(t)=Tr(t^{-1})$,
one can verify :

$|\pi^{\frac{s}{2}}+Tr(t)-\pi^{\frac{s}{2}}\rho N(t)|^{s_{1}}=|\pi^{\frac{s}{2}}N(\frac{1}{t})+Tr(\frac{1}{t})-\pi^{\frac{s}{2}}\rho|^{s_{1}}$

~

Setting this into \eqref{eq:7.3}:

~

\[
L(x,1,\chi_{2},s_{1},s_{2})=\frac{\chi_{2}(-\Delta)}{q^{-1}+1}\underset{O_{F}^{*}}{[\int}|\pi^{\frac{s}{2}}N(\frac{1}{t})+Tr(\frac{1}{t})-\pi^{\frac{s}{2}}\rho|^{s_{1}}dt+q^{-s\cdot s_{1}-1}+q^{-s\cdot s_{1}-1}].
\]

~

We substitute $u=\frac{1}{t}$ 

~

\begin{equation}
L(x,1,\chi_{2},s_{1},s_{2})=\frac{\chi_{2}(-\Delta)}{q^{-1}+1}\underset{O_{F}^{*}}{[\int}|\pi^{\frac{s}{2}}N(u)+Tr(u)-\pi^{\frac{s}{2}}\rho|^{s_{1}}du+2q^{-s\cdot s_{1}-1}].\label{eq:7.4}
\end{equation}

We use the following identity:

~

$\pi^{\frac{s}{2}}N(u)+Tr(u)=(\varpi^{\frac{s}{2}}u+\bar{\varpi}^{-\frac{s}{2}})(\overline{\varpi}^{\frac{s}{2}}\overline{u}+\varpi^{-\frac{s}{2}})-\pi^{-\frac{s}{2}}=N(\varpi^{\frac{s}{2}}u+\bar{\varpi}^{-\frac{s}{2}})-\pi^{-\frac{s}{2}}$

~

Setting this identity into \eqref{eq:7.4}:

~

\[
L(x,1,\chi_{2},s_{1},s_{2})=\frac{q\chi_{2}(-\Delta)}{q+1}\underset{O_{F}^{*}}{[\int}|N(\varpi^{\frac{s}{2}}u+\bar{\varpi}^{-\frac{s}{2}})-\pi^{-\frac{s}{2}}-\pi^{\frac{s}{2}}\rho|^{s_{1}}du+2q^{-s\cdot s_{1}-1}]=
\]

~

\[
\frac{\chi_{2}(-\Delta)}{q^{-1}+1}\underset{O_{F}^{*}}{[\int}|\pi^{\frac{s}{2}}N(\varpi^{\frac{s}{2}}u+\bar{p}^{-\frac{s}{2}})-\pi^{-\frac{s}{2}}-\pi^{\frac{s}{2}}\rho|^{s_{1}}du+2q^{-s\cdot s_{1}-1}]=
\]

~

\[
\frac{\chi_{2}(-\Delta)}{q^{-1}+1}\underset{O_{F}^{*}}{[\int}q^{-s\cdot s_{2}}|\pi^{-s}[N(\varpi^{s}u+1)-(1+\pi^{s}\rho)]|^{s_{1}}du+2q^{-s\cdot s_{1}-1}].
\]

~

Since $1+\pi^{s}\rho$ is not a norm, we have $[N(\varpi^{s}u+1)-(1+\pi^{s}\rho)]=\pi^{s}\eta,\,\,\,\eta\in O_{E}^{*}$
(see Section \ref{sub:The-Norm} ), we have:

~

\[
L(x,1,\chi_{2},s_{1},s_{2})=\frac{\chi_{2}(-\Delta)}{q^{-1}+1}\underset{O_{F}^{*}}{[\int}q^{-s\cdot s_{2}}du+2q^{-s\cdot s_{1}-1}]=
\]

~

\[
\frac{\chi_{2}(-\Delta)}{q^{-1}+1}[q^{-s\cdot s_{1}}(1-q^{-1})+2q^{-s\cdot s_{1}-1}]=\frac{\chi_{2}(-\Delta)}{q^{-1}+1}[q^{-s\cdot s_{1}}+q^{-s\cdot s_{1}-1}].
\]

~

Substituting $s\mapsto z$:

~

\[
L(x,1,\chi_{2},z)=\chi_{2}(-\Delta)q^{-2s\cdot z_{2}+\frac{s}{2}}q^{s\cdot(z_{1}+z_{2})}.
\]

~

\textbf{\uline{4.\label{6.2.1 matrix 4}}}

~

$x=\left(\begin{array}{cc}
\pi^{m} & 1\\
1 & 0
\end{array}\right)$ ~~~$0<m<\frac{s}{2}$

~

\[
L(x,1,\chi_{2},s_{1},s_{2})=\frac{\chi_{2}(-1)}{q^{-1}+1}\underset{O_{F}}{[\int}|\pi^{m}+Tr(t)|^{s_{1}}dt+\underset{\varpi O}{\int}|\pi^{m}N(y)+Tr(y)|^{s_{1}}dy].
\]

~

We have that $|Tr(t)|<|\pi^{m}|$ (see Section \eqref{sub:The-Trace}),
and so the first integrand is constant: 

~
\begin{equation}
L(x,1,\chi_{2},s_{1},s_{2})=\frac{\chi_{2}(-1)}{q^{-1}+1}[q^{-2s_{1}m}+\underset{\varpi O}{\int}|\pi^{m}N(y)+Tr(y)|^{s_{1}}dy].\label{eq:7.5}
\end{equation}

~

We use the following identity:

~

$\pi^{m}N(y)+Tr(y)=(\varpi^{-m}+\overline{\varpi}^{m}y)(\bar{\varpi}^{-m}+\varpi^{m}\bar{y})-\pi^{-m}=N(\varpi^{-m}+\overline{\varpi}^{m}y)-\pi^{-m}.$

~

We set it into \eqref{eq:7.5}:

~

~

\[
L(x,1,\chi_{2},s_{1},s_{2})=\frac{\chi_{2}(-1)}{q^{-1}+1}[q^{-2s_{1}m}+q^{2ms_{1}}\underset{\varpi O_{F}}{\int}|N(1+\pi^{m}y)-1|^{s_{1}}dy].
\]

~

We substitute $u=\pi^{m}y$ :

~

\[
L(x,1,\chi_{2},s_{1},s_{2})=\frac{\chi_{2}(-1)}{q^{-1}+1}[q^{-2s_{1}m}+q^{2ms_{1}+m}\underset{\varpi^{2m+1}O_{F}}{\int}|N(1+u)-1|^{s_{1}}du].
\]

We substitute $t=1+u$ :

~

\[
L(x,1,\chi_{2},s_{1},s_{2})=\frac{\chi_{2}(-1)}{q^{-1}+1}[q^{-2s_{1}m}+q^{2ms_{1}+2m}\underset{1+\varpi^{2m+1}O_{F}}{\int}|N(t)-1|^{s_{1}}dt].
\]

~

We use Lemma \ref{measure lemma}. The last integral splits to 3 different
sums:

~

$L(x,1,\chi_{2},s_{1},s_{2})=\frac{\chi_{2}(-1)}{q^{-1}+1}[q^{-2s_{1}m}+q^{2ms_{1}+2m}[\overset{\infty}{\underset{i=s+1}{\sum}2}(1-q^{-1})q^{-2is_{1}-i}+(q-2)q^{-s-1}q^{-2s\cdot s_{1}}+\underset{i=2m+1}{\overset{s-1}{\sum}}q^{-2is_{1}-i}(1-q^{-1})].$

~

~

After calculating the sums and substituting $s\mapsto z$, we have:

~

\[
L(x,1,\chi_{2},z)=\frac{\chi_{2}(-1)}{q^{-1}+1}[q^{m(2z_{1}-2z_{2}+1)}+\{(1-q^{-1})[\frac{2q^{(2z_{1}-2z_{2})(s+1)}}{1-q^{2z_{1}-2z_{2}}}+\frac{(q^{(2m+1)(2z_{1}-2z_{2})}-q^{s(2z_{1}-2z_{2})})}{1-q^{2z_{1}-2z_{2}}}]+
\]

\[
+(q-2)q^{m(2z_{2}-2z_{1}+1)}\frac{q^{s(2z_{1}-2z_{2})-1+2z_{2}}-q^{s(2z_{1}-2z_{2})-1+2z_{1}}}{q^{2z_{2}}-q^{2z_{1}}}\}].
\]

~

Simplifying this expression once more, we get:

~

\[
L(x,1,\chi_{2},z)=\frac{\chi_{2}(-1)}{q^{-1}+1}\frac{q^{m-2sz_{2}}(1-q^{-1})}{q^{2z_{2}}-q^{2z_{1}}}\underset{\sigma}{\sum}\sigma(q^{\langle(m,s-m+1),2z\rangle}).
\]

~

\textbf{\uline{~5.}}

~

$x=\left(\begin{array}{cc}
\pi^{m} & \varpi\\
\overline{\varpi} & 0
\end{array}\right)$~$\,0<m<\frac{s}{2}+1$

~

\[
L(x,1,\chi_{2},s_{1},s_{2})=\frac{\chi_{2}(-1)q^{-2s_{2}}}{q^{-1}+1}\underset{O_{F}}{[\int}|\pi^{m}+Tr(\overline{\varpi}t)|^{s_{1}}dt+\underset{\varpi O_{F}}{\int}|\pi^{m}N(y)+Tr(\overline{\varpi}y)|^{s_{1}}dy].
\]

Once again, note that $|\pi^{m}+Tr(\overline{\varpi}t)|^{s_{1}}=q^{-2s_{1}m}$
for any $t\in O_{F}$ . (see Section \ref{sub:The-Trace}). 

We substitute $u=\varpi y$:

~

\[
L(x,1,\chi_{2},s_{1},s_{2})=\frac{\chi_{2}(-1)q^{-2s_{2}}}{q^{-1}+1}[q^{-2s_{1}m}+q\underset{\varpi^{2}O_{F}}{\int}|\pi^{m-1}N(u)+Tr(u)|^{s_{1}}du].
\]

~

We use the identity: $\pi^{m-1}N(u)+Tr(u)=N(\varpi^{1-m}+\bar{\varpi}^{m-1}u)-\pi^{1-m}$
to get:

~

~

\[
L(x,1,\chi_{2},s_{1},s_{2})=\frac{\chi_{2}(-1)q^{-2s_{2}}}{q^{-1}+1}[q^{-2ms_{1}}+q^{1+(m-1)2s_{1}}\underset{\varpi^{2}O_{F}}{\int}|N(1+\pi^{m-1}u)-1|^{s_{1}}du]
\]

~

Substitute $y=\pi^{m-1}u$:

~

\[
L(x,1,\chi_{2},s_{1},s_{2})=\frac{\chi_{2}(-1)q^{-2s_{2}}}{q^{-1}+1}[q^{-2ms_{1}}+q^{2m-1+(m-1)2s_{1}}\underset{\varpi^{2+2(m-1)}O_{F}}{\int}|N(1+y)-1|^{s_{1}}dy].
\]

~

Substitute $t=1+y$:

~

\[
L(x,1,\chi_{2},s_{1},s_{2})=\frac{\chi_{2}(-1)q^{-2s_{2}}}{q^{-1}+1}[q^{-2ms_{1}}+q^{2m-1+(m-1)2s_{1}}\underset{1+\varpi^{2m}O_{F}}{\int}|N(t)-1|^{s_{1}}dt].
\]

~

We calculate the last integral by using Lemma \ref{measure lemma}
.

~

\[
L(x,1,\chi_{2},s_{1},s_{2})=\frac{\chi_{2}(-1)q^{-2s_{2}}}{q^{-1}+1}\{q^{-2ms_{1}}+q^{m+(m-1)+(m-1)2s_{1}}\times
\]

\[
[\overset{\infty}{\underset{i=s+1}{\sum}2}(1-q^{-1})q^{-2is_{1}-i}+(q-2)q^{-s-1}q^{-2s\cdot s_{1}}+\underset{i=2m}{\overset{s-1}{\sum}}q^{-2is_{1}-i}(1-q^{-1})]\}=
\]

~

Substituting $s\mapsto z$ and simplifying the expression, we get: 

~

\[
L(x,1,\chi_{2},z)=\frac{\chi_{2}(-1)q^{-2s_{2}}}{q^{-1}+1}\{q^{-2ms_{1}}+q^{m+(m-1)(2z_{2}-2z_{1})}\times
\]

~

\[
[(1-q^{-1})(\frac{2q^{(2z_{1}-2z_{2})(s+1)+2z_{2}}}{q^{2z_{2}}-q^{2z_{1}}}+\frac{q^{(2z_{1}-2z_{2})2m+2z_{2}}-q^{(2z_{1}-2z_{2})s+2z_{2}}}{q^{2z_{2}}-q^{2z_{1}}})+(q-2)q^{s(2z_{1}-2z_{2})-1}]\}
\]

~

~

\[
=\frac{\chi_{2}(-1)q^{2z_{2}-\frac{1}{2}+m}}{q^{-1}+1}[q^{2z_{1}-2z_{2}}+q^{(1-m)(2z_{1}-2z_{2})}[\frac{2(q-1)q^{(2z_{1}-2z_{2})(s+1)+2z_{2}-1}}{q^{2z_{2}}-q^{2z_{1}}}+
\]

\[
\frac{(q-1)q^{(2z_{1}-2z_{2})2m+2z_{2}-1}-(q-1)q^{(2z_{1}-2z_{2})s+2z_{2}-1}}{q^{2z_{2}}-q^{2z_{1}}}++\frac{(q-2)q^{s(2z_{1}-2z_{2})-1+2z_{2}}-(q-2)q^{s(2z_{1}-2z_{2})-1+2z_{1}}}{q^{2z_{2}}-q^{2z_{1}}}].
\]

~

We extract and simplify the different elements in this expression: 

~

~

\[
L(x,1,\chi_{2},z)=\frac{\chi_{2}(-1)q^{-2sz_{2}-\frac{1}{2}+m}}{q^{-1}+1}\times[\frac{q^{m2z_{1}+(2-m+s)2z_{2}}}{q^{2z_{2}}-q^{2z_{1}}}+\frac{q^{(s+2-m)2z_{1}+m2z_{2}}}{q^{2z_{2}}-q^{2z_{1}}}
\]

\[
+\frac{-q^{(m+1)2z_{1}+(1-m+s)2z_{2}-1}-q^{(s+1-m)2z_{1}+(1+m)2z_{2}-1}}{q^{2z_{2}}-q^{2z_{1}}}]
\]

~

Note that $q^{2sz_{2}}L(x,1,\chi_{2},z)$ is an antisymmetric function.
We can express $L(x,1,\chi_{2},z)$ as follows:

~

\[
L(x,1,\chi_{2},z)=\frac{\chi_{2}(-1)q^{-\frac{1}{2}}}{q^{-1}+1}\frac{q^{m-2sz_{2}}}{q^{2z_{2}}-q^{2z_{1}}}[\underset{\sigma}{\sum}\sigma(q^{\langle(m,2+s-m),2z\rangle}-q^{\langle(m+1,s+1-m),2z\rangle-1})]
\]

~

~

\textbf{\uline{6.}}

$x\sim\left(\begin{array}{cc}
\pi^{m} & 1\\
1 & -\pi^{s-m}\rho
\end{array}\right)$,~~~$0<m<\frac{s}{2}$

~

\[
L(x,1,\chi_{2},s_{1},s_{2})=\frac{q\chi_{2}(-\Delta)}{q+1}\underset{O_{F}}{[\int}|\pi^{m}+Tr(t)-\pi^{s-m}\rho N(t)|^{s_{1}}dt+\underset{\varpi O_{F}}{\int}|\pi^{m}N(y)+Tr(y)-\pi^{s-m}\rho|^{s_{1}}dy].
\]

~

Because $|Tr(O_{F})|\leq|\pi^{\frac{s}{2}}|$ (see Section \ref{sub:The-Trace})
and $|\pi^{s-m}|<|\pi|^{m}$ , we have that 

the first integrand is constant:

~

\begin{equation}
L(x,1,\chi_{2},s_{1},s_{2})=\frac{q\chi_{2}(-\Delta)}{q+1}[q^{-2m\cdot s_{1}}+\underset{\varpi O_{F}}{\int}|\pi^{m}N(y)+Tr(y)-\pi^{s-m}\rho|^{s_{1}}dy].\label{eq:7.6}
\end{equation}

We use the identity $\pi^{m}N(y)+Tr(y)=N(\varpi^{-m}+\varpi^{m}y)-\pi^{-m}$
and set it into Eq \ref{eq:7.6}:

~

\[
L(x,1,\chi_{2},s_{1},s_{2})=\frac{\chi_{2}(-\Delta)}{q^{-1}+1}[q^{-2m\cdot s_{1}}+\underset{\varpi O_{F}}{\int}|N(\varpi^{-m}+\varpi^{m}y)-\pi^{-m}-\pi^{s-m}\rho|^{s_{1}}dy]=
\]

~

\[
\frac{\chi_{2}(-\Delta)}{q^{-1}+1}[q^{-2m\cdot s_{1}}+q^{2ms_{1}}\underset{\varpi O_{F}}{\int}|N(1+\pi^{m}y)-(1+\pi^{s}\rho)|^{s_{1}}dy].
\]

~

Substituting $t=1+\pi^{m}y$:

~

~

\begin{equation}
L(x,1,\chi_{2},s_{1},s_{2})=\frac{\chi_{2}(-\Delta)}{q^{-1}+1}[q^{-2m\cdot s_{1}}+q^{2ms_{1}+2m}\underset{1+\varpi^{2m+1}O_{F}}{\int}|N(t)-(1+\pi^{s}\rho)|^{s_{1}}dt].\label{eq:7.7-1}
\end{equation}

~

We calculate the last integral by decomposing $1+\varpi^{2m+1}O_{F}=\underset{j=2m+1}{\overset{s-1}{\bigcup}}(1+\varpi^{j}O_{F}^{*})\cup(1+\varpi^{s}O_{F}).$ 

From Section \ref{sub:The-Norm}, we deduce that for any $t\in1+p^{i}O_{F}^{*}$and
$2m+1\leq i<s$ the value of $|N(t)-(1+\pi^{s}\rho)|^{s_{1}}$is $q^{-2is_{1}}$. 

~

Because $1+\pi^{s}\rho\not\notin N(F)$, in the subgroup $1+\varpi^{s}O_{F}$
the integrand is simply $q^{-2s\cdot s_{1}}.$

~

Altogether, we have:

~

\[
L(x,1,\chi_{2},s_{1},s_{2})=\frac{\chi_{2}(-\Delta)}{q^{-1}+1}[q^{-2m\cdot s_{1}}+q^{2ms_{1}+2m}(\overset{s-1}{\underset{i=2m+1}{\sum}}(1-q^{-1})q^{-i}q^{-i2s_{1}}+q^{-s}q^{-2s\cdot s_{1}})]=
\]

~

\[
\frac{\chi_{2}(-\Delta)}{q^{-1}+1}[q^{-2m\cdot s_{1}}+q^{2ms_{1}+2m}(\overset{s-1}{\underset{i=2m+1}{\sum}}q^{(-2s_{1}-1)i}+q^{s(-2s_{1}-1)})]
\]

~

We simplify the sum and substitute $s\mapsto z$ to get: 

~
\[
L(x,1,\chi_{2},s_{1},s_{2})=\frac{\chi_{2}(-\Delta)}{q^{-1}+1}[\frac{q^{m(2z_{1}-2z_{2}+1)+2z_{2}}-q^{m\cdot(2z_{1}-2z_{2}+1)+2z_{1}}}{q^{2z_{2}}-q^{2z_{1}}}+
\]

\[
q^{2ms_{1}+2m}(\frac{q^{(2z_{1}-2z_{2})(2m+1)}-q^{(2z_{1}-2z_{2})s}}{1-q^{2z_{1}-2z_{2}}}+q^{s(2z_{1}-2z_{2})})]=
\]

~

\[
\frac{\chi_{2}(-\Delta)q^{m}q^{-2sz_{2}}}{q^{-1}+1}[\frac{q^{m2z_{1}+(1-m+s)2z_{2}}}{q^{2z_{2}}-q^{2z_{1}}}+\frac{q^{(s+1-m)2z_{1}+m2z_{2}}}{q^{2z_{1}}-q^{2z_{2}}}]
\]

~

And after further manipulations:

~

\[
L(x,1,\chi_{2},z)=\frac{\chi_{2}(-\Delta)q^{m-2sz_{2}}}{q^{-1}+1}[\underset{\sigma}{\sum}\sigma(\frac{q^{\langle(m,1-m+s),2z\rangle}}{q^{2z_{2}}-q^{2z_{1}}})].
\]

~

\textbf{\uline{7.}}

$x\sim\left(\begin{array}{cc}
\pi^{m} & p\\
\overline{p} & -\pi^{s+1-m}\rho
\end{array}\right)$,~~~~$0<m<\frac{s}{2}+1.$

~

$L(x,1,\chi_{2},z)=\frac{\chi_{2}(-\Delta)q^{-2s_{2}}}{q^{-1}+1}\underset{O_{F}}{[\int}|\pi^{m}+Tr(\bar{\varpi}t)-\pi^{s+1-m}\rho N(t)|^{s_{1}}dt+\underset{\varpi O_{F}}{\int}|\pi^{m}N(y)+Tr(\bar{\varpi}y)-\pi^{s+1-m}\rho|^{s_{1}}dy]$

~

Similarly to the previous calculations the first integrand is also
constant, so

~

\[
L(x,1,\chi_{2},s_{1},s_{2})=\frac{\chi_{2}(-\Delta)q^{-2s_{2}}}{q^{-1}+1}[q^{-2ms_{1}}+\underset{\varpi O_{F}}{\int}|\pi^{m-1}N(\bar{\varpi}y)+Tr(\bar{\varpi}y)-\pi^{s+1-m}\rho|^{s_{1}}dy].
\]

~

We substitute $u=\varpi y$:

~

\begin{equation}
L(x,1,\chi_{2},s_{1},s_{2})=\frac{\chi_{2}(-\Delta)q^{-2s_{2}}}{q^{-1}+1}[q^{-2ms_{1}}+q\underset{\varpi^{2}O_{F}}{\int}|\pi^{m-1}N(u)+Tr(u)-\pi^{s+1-m}\rho|^{s_{1}}du].\label{eq:7.8}
\end{equation}

~

We use the identity: $\pi^{m-1}N(u)+Tr(u)=N(\varpi^{1-m}+\varpi^{m-1}u)-\pi^{1-m}$
and set it to Eq \ref{eq:7.8}:

~

~

\[
L(x,1,\chi_{2},s_{1},s_{2})=\frac{\chi_{2}(-\Delta)q^{-2s_{2}}}{q^{-1}+1}[q^{-2ms_{1}}+q\underset{\varpi^{2}O_{F}}{\int}|N(\varpi^{1-m}+\varpi^{m-1}u)-\pi^{1-m}-\pi^{s+1-m}\rho|^{s_{1}}du]=
\]

~

\[
\frac{\chi_{2}(-\Delta)q^{-2s_{2}}}{q^{-1}+1}[q^{-2ms_{1}}+q^{1+(m-1)2s_{1}}\underset{\varpi^{2}O_{F}}{\int}|N(1+\varpi^{2m-2}u)-(1+\pi^{s}\rho)|^{s_{1}}du].
\]

~

We substitute $t=1+\varpi^{2m-2}u$ :

~

\[
L(x,1,\chi_{2},s_{1},s_{2})=\frac{\chi_{2}(-\Delta)q^{-2s_{2}}}{q^{-1}+1}[q^{-2ms_{1}}+q^{(m-1)2s_{1}+2m-1}\underset{1+\varpi^{2m}O_{F}}{\int}|N(t)-(1+\pi^{s}\rho)|^{s_{1}}dt].
\]

~

The calculation of the integral is similar to the calculation of the
integral in Eq \ref{eq:7.7-1}:

~

\[
L(x,1,\chi_{2},s_{1},s_{2})=\frac{\chi_{2}(-\Delta)q^{-2s_{2}}}{q^{-1}+1}[q^{-2ms_{1}}+q^{(m-1)(2z_{2}-2z_{1})+m}(\overset{s-1}{\underset{i=2m}{\sum}}q^{(-2s_{1}-1)i}+q^{-s}q^{-2s\cdot s_{1}})].
\]

~

We substitute $s\mapsto z$:

~

\[
L(x,1,\chi_{2},z)=\frac{\chi_{2}(-\Delta)q^{2z_{2}-\frac{1}{2}+m-(s+1)2z_{2}}}{q^{-1}+1}[\frac{q^{m2z_{1}+(2+s-m)2z_{2}+2z_{2}}}{q^{2z_{2}}-q^{2z_{1}}}-\frac{q^{(s-m+2)2z_{1}+m2z_{2}}}{q^{2z_{2}}-q^{2z_{1}}}].
\]

~

~

\[
=\frac{\chi_{2}(-\Delta)q^{m-\frac{1}{2}-s2z_{2}}}{q^{-1}+1}[\underset{\sigma}{\sum}\sigma(\frac{q^{\langle(m,2+s-m),2z\rangle}}{q^{2z_{2}}-q^{2z_{1}}})].
\]

~

\subsubsection{RU case}

~

Recall that $l=\frac{s-1}{2}$ , $Tr(\varpi^{2i}O_{F})=\pi^{l+1+i}O_{E}$
and $Tr(\varpi^{2i-1}O_{F})=\pi^{l+i}O_{E}$.

~

\textbf{\uline{1.}}

$x\sim\left(\begin{array}{cc}
\pi^{\frac{s+1}{2}} & \varpi\\
\overline{\varpi} & -\pi^{\frac{s+1}{2}}\rho
\end{array}\right)$

~

$L(x,1,\chi_{2},s_{1},s_{2})=\frac{\chi_{2}(-\Delta)q^{-2s_{2}}}{q^{-1}+1}\underset{O_{F}}{[\int}|\pi^{\frac{s+1}{2}}+Tr(\bar{\varpi}t)-\pi^{\frac{s+1}{2}}\rho N(t)|^{s_{1}}dt+\underset{\varpi O_{F}}{\int}|\pi^{\frac{s+1}{2}}N(y)+Tr(\bar{\varpi}y)-\pi^{\frac{s+1}{2}}\rho|^{s_{1}}dy]$

~

Similarly to previous calculation, the second integrand is constant
because $|\pi^{\frac{s+1}{2}}\rho|>|\pi^{\frac{s+1}{2}}N(y)+Tr(\bar{\varpi}y)|$
(see Section \ref{sub:The-Norm}) . 

~

\[
L(x,1,\chi_{2},s_{1},s_{2})=\frac{\chi_{2}(-\Delta)q^{-2s_{2}}}{q^{-1}+1}\underset{O_{F}}{[\int}|\pi^{\frac{s+1}{2}}+Tr(\bar{\varpi}t)-\pi^{\frac{s+1}{2}}\rho N(t)|^{s_{1}}dt+q^{(s+1)(-s_{1})-1}].
\]

~

We have that: 
\[
\underset{O_{F}}{\int}|\pi^{\frac{s+1}{2}}+Tr(\bar{\varpi}t)-\pi^{\frac{s+1}{2}}\rho N(t)|^{s_{1}}dt=\underset{O_{F}^{*}}{\int}|\pi^{\frac{s+1}{2}}+Tr(\bar{\varpi}t)-\pi^{\frac{s+1}{2}}\rho N(t)|^{s_{1}}dt+.
\]

\[
\underset{\varpi O_{F}}{\int}|\pi^{\frac{s+1}{2}}+Tr(\bar{\varpi}t)-\pi^{\frac{s+1}{2}}\rho N(t)|^{s_{1}}dt=
\]

\[
\underset{O_{F}^{*}}{\int}|\pi^{\frac{s+1}{2}}+Tr(\bar{\varpi}t)-\pi^{\frac{s+1}{2}}\rho N(t)|^{s_{1}}dt+q^{(s+1)(-s_{1})-1}.
\]

~

So :

~

\[
L(x,1,\chi_{2},s_{1},s_{2})=\frac{\chi_{2}(-\Delta)q^{-2s_{2}}}{q^{-1}+1}\underset{O_{F}^{*}}{[\int}|\pi^{\frac{s+1}{2}}+Tr\bar{(\varpi}y)-\pi^{\frac{s+1}{2}}\rho N(y)|^{s_{1}}dy+2q^{(s+1)(-s_{1})-1}]=
\]

\[
\frac{\chi_{2}(-\Delta)q^{-2s_{2}}}{q^{-1}+1}\underset{O_{F}^{*}}{[\int}|\pi^{\frac{s+1}{2}}N(\frac{1}{y})+Tr\bar{(\frac{\varpi}{y}})-\pi^{\frac{s+1}{2}}\rho|^{s_{1}}dy+2q^{(s+1)(-s_{1})-1}]
\]

~

We substitute $u=\frac{1}{y}:$

~

\begin{equation}
L(x,1,\chi_{2},s_{1},s_{2})=\frac{\chi_{2}(-\Delta)q^{-2s_{2}}}{q^{-1}+1}\underset{O_{F}^{*}}{[\int}|\pi^{\frac{s+1}{2}}N(u)+Tr(\overline{\varpi}u)-\pi^{\frac{s+1}{2}}\rho|^{s_{1}}du+2q^{(s+1)(-s_{1})-1}].\label{eq:7.9}
\end{equation}

~

We use the identity: $\pi^{\frac{s+1}{2}}N(u)+Tr(pu)=N((\varpi^{\frac{s+1}{2}}u+\overline{\varpi^{\frac{1-s}{2}}})-\pi^{\frac{-s-1}{2}}$
and set it to Eq \ref{eq:7.9}:

~

~

\[
L(x,1,\chi_{2},z)=\frac{q\chi_{2}(-\Delta)q^{-2s_{2}}}{q+1}\underset{O_{F}^{*}}{[\int}|N((\varpi^{\frac{s+1}{2}}u+\overline{\varpi^{\frac{1-s}{2}}})-\pi^{\frac{-s+1}{2}}-\pi^{\frac{s+1}{2}}\rho)|^{s_{1}}du+2q^{(s+1)(-s_{1})-1}]=
\]

~

\[
\frac{q\chi_{2}(-\Delta)q^{-2s_{2}}}{q+1}[\underset{O_{F}^{*}}{\int}|\pi^{\frac{1-s}{2}}N((\varpi^{s}u+1)-\pi^{\frac{1-s}{2}}-\pi^{\frac{s+1}{2}}\rho)|^{s_{1}}du+2q^{(s+1)(-s_{1})-1}]
\]

~

\[
=\frac{\chi_{2}(-\Delta)q^{-2s_{2}}}{q^{-1}+1}[\underset{O_{F}^{*}}{\int}|\pi^{\frac{1-s}{2}}[N((\varpi^{s}u+1)-(1+\pi^{s}\rho)]|^{s_{1}}du+2q^{(s+1)(-s_{1})-1}]
\]

~

But, as in the RP case, we have:~~ $N(\varpi^{s}u+1)-(1+\pi^{s}\rho)=\pi^{s}\eta\,\,\,,\,\,\eta\in O_{E}^{*}$

~

~

\[
L(x,1,\chi_{2},z)=\frac{\chi_{2}(-\Delta)q^{-2s_{2}}}{q^{-1}+1}[\underset{O_{F}^{*}}{\int}|\pi^{\frac{1+s}{2}}\eta|^{s_{1}}du+2q^{(s+1)(-s_{1})-1}]=
\]

~

\[
\frac{\chi_{2}(-\Delta)q^{-2s_{2}}}{q^{-1}+1}[q^{(s+1)(-s_{1})}(1-q^{-1})+2q^{(s+1)(-s_{1})-1}]
\]

~

Substitute $s\mapsto z:$

~

\[
L(x,1,\chi_{2},z)=\chi_{2}(-\Delta)q^{-2s\cdot z_{2}+\frac{s}{2}}[q^{(s+1)(z_{1}+z_{2})}].
\]

~

\textbf{\uline{2.}}

$x\sim\left(\begin{array}{cc}
 & 1\\
1
\end{array}\right)$

~

$L(x,1,\chi_{2},s_{1},s_{2})=\frac{\chi_{2}(-1)}{q^{-1}+1}\underset{O_{F}}{[\int}|Tr(t))|^{s_{1}}dt+\underset{\varpi O_{F}}{\int}|Tr(y)|^{s_{1}}dy]$

~

Similarly to the RP case , we use \ref{sub:The-Trace} to the substitution
$u=Tr(t),\,\, v=Tr(y)$ . We have:

~

\[
L(x,1,\chi_{2},s_{1},s_{2})=\frac{\chi_{2}(-1)}{q^{-1}+1}[q^{l+1}\underset{\pi^{l+1}O_{E}}{\int}|u|^{s_{1}}du+q^{l}\underset{\pi^{l+1}O_{E}}{\int}|v|^{s_{1}}dv]=
\]

~

Calculation on the integrals above results in geometric sums:

~

\[
L(x,1,\chi_{2},s_{1},s_{2})=\frac{\chi_{2}(-1)}{q^{-1}+1}[q\underset{i=1}{\overset{\infty}{\sum}}q^{-2(l+i)s_{1}}\cdot q^{-i}(1-q^{-1})+\underset{i=1}{\overset{\infty}{\sum}}q^{-2(l+i)s_{1}}\cdot q^{-i}(1-q^{-1})]=
\]

~

\[
\frac{\chi_{2}(-1)}{q^{-1}+1}[\underset{i=1}{\overset{\infty}{\sum}}q^{-2(l+i)s_{1}}\cdot q^{-i}(1-q^{-1})(q+1)]=q(1-q^{-1})\chi_{2}(-1)q^{-2ls_{1}}[\underset{i=1}{\overset{\infty}{\sum}}q^{(-2s_{1}-1)i}].
\]

~

We substitute $s\mapsto z$:

~

\[
L(x,1,\chi_{2},z)=q\chi_{2}(-1)(1-q^{-1})q^{-2sz_{2}}q^{l(2z_{1}+2z_{2})}[\frac{q^{2z_{1}+2z_{2}}}{q^{2z_{2}}-q^{2z_{1}}}]=
\]

\[
q\chi_{2}(-1)(1-q^{-1})q^{-2sz_{2}}[\frac{q^{(s+1)(z_{1}+z_{2})}}{q^{2z_{2}}-q^{2z_{1}}}]
\]

~

\textbf{\uline{3.}}

$x\sim\left(\begin{array}{cc}
 & \varpi\\
\overline{\varpi}
\end{array}\right)$

\[
L(x,1,\chi_{2},s_{1},s_{2})=\frac{\chi_{2}(-1)q^{-2s_{2}}}{1+q^{-1}}[\underset{O_{F}}{\int}|Tr(\overline{\varpi}t))|^{s_{1}}dt+\underset{\varpi O_{F}}{\int}|Tr(\overline{\varpi}y)|^{s_{1}}dy=
\]

~

\[
\frac{\chi_{2}(-1)q^{-2s_{2}+1}}{1+q^{-1}}\underset{\varpi O_{F}}{[\int}|Tr(t))|^{s_{1}}dt+\underset{\varpi^{2}O_{F}}{\int}|Tr(y)|^{s_{1}}dy]
\]

~

We substitute $u=Tr(t),\, v=Tr(y)$: (See Section \ref{sub:The-Trace}
and Lemma \ref{pushforword lemma})

~

\[
L(x,1,\chi_{2},s_{1},s_{2})=\frac{q^{l}\chi_{2}(-1)q^{-2s_{2}+1}}{1+q^{-1}}[\underset{\pi^{l+1}O_{E}}{\int}|u|^{s_{1}}du+\underset{\pi^{l+2}O_{E}}{\int}|v|^{s_{1}}dv]=
\]

~

\[
\frac{\chi_{2}(-1)q^{-2s_{2}+1}}{1+q^{-1}}[\underset{i=1}{\overset{\infty}{\sum}}q^{-2(l+i)s_{1}}\cdot q^{-i}(1-q^{-1})+\underset{i=2}{\overset{\infty}{\sum}}q^{-2(l+i)s_{1}}\cdot q^{-i}(1-q^{-1})]=
\]

~

We compute the geometric sums and simplify:

~

\[
L(x,1,\chi_{2},s_{1},s_{2})=\frac{\chi_{2}(-1)q^{-2s_{2}+1}q^{-2ls_{1}}(1-q^{-1})}{1+q^{-1}}[\frac{q^{(-2s_{1}-1)}}{1-q^{-2s_{1}-1}}+\frac{q^{(-2s_{1}-1)2}}{1-q^{-2s_{1}-1}}]=
\]

Substitute $s\mapsto z$:

\[
L(x,1,\chi_{2},z)=\frac{\chi_{2}(-1)q^{\frac{3}{2}}q^{l(2z_{1}-2z_{2}+1)}(1-q^{-1})}{1+q^{-1}}[\frac{q^{2z_{1}+2z_{2}}}{q^{2z_{2}}-q^{2z_{1}}}+\frac{q^{4z_{1}}}{q^{2z_{2}}-q^{2z_{1}}}]=
\]

~

\[
\frac{\chi_{2}(-1)(1-q^{-1})q^{l+\frac{1}{2}}q^{-2s\cdot z_{2}}}{1+q^{-1}}q^{l(2z_{1}+2z_{2})}q^{2z_{1}+2z_{2}}[\frac{q^{2z_{1}}+q^{2z_{2}}}{q^{2z_{2}}-q^{2z_{1}}}]
\]

~

~

\section{~Calculation of $L(x,\chi^{*},\chi_{2},z)$ }

~

In the following two subsections we prove that on most of the representatives
the spherical function vanish.

~

For this section, $A_{i},B_{i},C_{i},..$ will denote constants that
are resulted from different coordinate transformation. 

~

\subsection{Calculation on non-diagonal representatives}

~

Recall that for $x\sim\left(\begin{array}{cc}
a & \varpi^{c}\\
\overline{\varpi}^{c} & b
\end{array}\right)$, we have from \ref{formula lemma} :

~

$L(x,1,\chi_{2},s_{1},s_{2})=\frac{q\chi_{2}(det(x))|det(x)|^{s_{2}}}{q+1}\underset{O_{F}}{[\int}|a+Tr(\overline{\varpi}^{c}t)+bN(t)|^{s_{1}}dt+\underset{\varpi O_{F}}{\int}|aN(y)+Tr(\overline{\varpi}^{c}y)+b|^{s_{1}}dy].$

~

We will denote :

\[
\begin{array}{c}
I_{1}=\underset{O_{F}}{\int}|a+Tr(\overline{\varpi}^{c}t)+bN(t)|^{s_{1}}dt\\
I_{2}=\underset{\varpi O_{F}}{\int}|aN(y)+Tr(\overline{\varpi}^{c}y)+b|^{s_{1}}dy.
\end{array}
\]

~

\subsubsection{Common Representatives for RP and RU }

~

\textbf{\uline{1.}} $x\sim\left(\begin{array}{cc}
0 & \varpi^{a}\\
\overline{\varpi}^{a} & 0
\end{array}\right)$

~

\[
L(x,\chi^{*},\chi_{2},s)=\frac{q\chi_{2}(1)}{q+1}[\underset{O_{F}}{\int}\chi^{*}(Tr(\overline{\varpi}^{a}t))|Tr(\overline{\varpi}^{a}t)|^{s_{1}}dt+\underset{\varpi O_{F}}{\int}\chi^{*}(Tr(\overline{\varpi}^{a}t))|Tr(\overline{\varpi}^{a}t)|^{s_{1}}dt]
\]

~

After substitution $u=Tr(\varpi^{a}t)$ (see Lemma \ref{sub:The-Trace},
Lemma \ref{pushforword lemma}) we have:

~

\[
\underset{\varpi^{k}O_{F}}{\int}\chi^{*}(Tr(\overline{\varpi}^{a}t))|Tr(\overline{\varpi}^{a}t)|^{s_{1}}dt=A_{1}\underset{\pi^{h}O_{E}}{\int}\chi^{*}(u)|u|^{s_{1}}du=A_{1}\underset{j=h}{\overset{\infty}{\sum}}\underset{\pi^{j}O_{E}^{*}}{\int}\chi^{*}(u)|u|^{s_{1}}du
\]

~

Where $h$ is defined by $Tr(\bar{\varpi}^{a+k}O_{F})=\pi^{h}O_{E}.$

~

But: $\underset{\pi^{j}O_{E}^{*}}{\int}\chi^{*}(u)|u|^{s_{1}}du=q^{-2js_{1}-j}\underset{O_{E}^{*}}{\int}\chi^{*}(y)dy^{*}=0$,
since $\chi^{*}$ is a non trivial character on $O_{E}^{*}$, the
integral on each term vanishes, so $L(x,\chi^{*},\chi_{2},z)=0$

~

~

~

\textbf{\uline{2.}} $x=\left(\begin{array}{cc}
\pi^{m} & 1\\
1 & -\pi^{s-m}\rho
\end{array}\right)$,~~~$0<m<\frac{s}{2}$ 
\begin{itemize}
\item We show that $I_{1}=0$:
\[
I_{1}=\underset{O_{F}}{\int}\chi^{*}(\pi^{m}+Tr(t)-\pi^{s-m}\rho N(t))|\pi^{m}+Tr(t)-\pi^{s-m}\rho N(t)|^{s_{1}}dt
\]

\end{itemize}
~

Note that $|Tr(t)-\pi^{s-m}\rho N(t)|<|\pi^{m}|$ (see Section \ref{corolary 3.3})
and so the absolute value is constant:

~

\[
I_{1}=A_{1}\underset{O_{F}}{\int}\chi^{*}(\pi^{m}+Tr(t)-\pi^{s-m}\rho N(t))dt
\]

~

We use the identity: $\pi^{m}+Tr(t)-\pi^{s-m}\rho N(t)=\pi^{m}+\frac{\pi^{m-s}}{\rho}-\pi^{s-m}\rho N(t-\frac{\pi^{m-s}}{\rho})$
to have: 

~

~

\[
I_{1}=\underset{O_{F}}{\int}\chi^{*}(\pi^{m}+\frac{\pi^{m-s}}{\rho}-\pi^{s-m}\rho N(t-\frac{\pi^{m-s}}{\rho}))dt
\]

~

Denote by $A_{2}=\chi^{*}(\frac{\pi^{m-s}}{\rho})$ , then we have:

~

\[
I_{1}=A_{2}\underset{O_{F}}{\int}\chi^{*}(1+\pi^{s}\rho-\pi^{2(s-m)}\rho^{2}N(t-\frac{\pi^{m-s}}{\rho}))dt=A_{2}\underset{O_{F}}{\int}\chi^{*}(1+\pi^{s}\rho-N(\pi^{s-m}\rho t-1))dt.
\]

We Substitute $u=\pi^{s-m}\rho t-1$, we have:

~

~

\[
I_{1}=A_{3}\underset{1+\varpi^{2(s-m)}O_{F}}{\int}\chi^{*}(1+\pi^{s}\rho-N(u))du.
\]

~

~

We substitute $\nu=N(u)$, note that $N(1+\varpi^{2(s-m)}O_{F})=1+\pi^{\frac{s}{2}-m+s}O_{E}$
(see \ref{sub:The-Norm}). We make use of Lemma \ref{pushforword lemma}
:

~

~

we have ~

~

\[
I_{1}=A_{4}\underset{1+\pi^{\frac{s}{2}-m+s}O_{E}}{\int}\chi^{*}[(1+\pi^{s}\rho)-t]dt.
\]

~

We substitute $s=(1+\pi^{s}\rho)-t$. Note that $\pi^{s}\rho-\pi^{\frac{s}{2}-m+s}O_{E}=\pi^{s}\rho(1+\pi^{\frac{s}{2}-m}O_{E})$,
so for the substitution $\pi^{s}\rho y=s$ we get:

~

\[
I_{1}=A_{4}\underset{\pi^{s}\rho(1+\pi^{\frac{s}{2}-m}O_{E})}{\int}\chi^{*}(s)ds=A_{5}\underset{1+\pi^{\frac{s}{2}-m}O_{E}}{\int}\chi^{*}(y)dy
\]

~

Because $1+\pi^{\frac{s}{2}-m}O_{F}$ contains non norm elements,
the character $\chi^{*}$ is a non trivial character of the group
$1+\pi^{\frac{s}{2}-m}O_{E}$ and the integral vanish.

~

~
\begin{itemize}
\item Now we show that:
\end{itemize}
\[
I_{2}=\underset{\varpi O_{F}}{\int}\chi^{*}(\pi^{m}N(y)+Tr(y)-\pi^{s-m}\rho)|\pi^{m}N(y)+Tr(y)-\pi^{s-m}\rho|^{s_{1}}dy=0
\]

We have already shown in our calculation of this representative in
Section 7.2 that:

~

\[
I_{2}=\underset{\varpi O}{\int}\chi^{*}(\pi^{m}N(y)+Tr(y)-\pi^{s-m}\rho)|\pi^{m}N(y)+Tr(y)-\pi^{s-m}\rho|^{s_{1}}dy=
\]

~

\[
B_{1}\underset{1+\varpi^{m+1}O_{F}}{\int}\chi^{*}(N(t)-(1+\pi^{s}\rho))|N(t)-(1+\pi^{s}\rho)|^{s_{1}}dt]
\]

~

The Norm induces an homomorphism :

\[
\widetilde{N}:\frac{1+\varpi^{m+1}O_{F}}{1+\varpi^{s+1}O_{F}}\to\frac{1+\pi^{m+1}O_{F}}{1+\pi^{s+1}O_{F}}
\]

~

On each coset the $|N(t)-(1+\pi^{s}\rho)|^{s_{1}}$ is a constant
function. (By our corollaries in Section \ref{sub:The-Norm})

Note that $q^{-s}<|N(t)-(1+\pi^{s}\rho)|_{E}\leq q^{-m-1}$ since
$N(t)\in1+\pi^{m+1}O_{E}$ and $1+\pi^{s}\rho\notin N(F^{*})$.

~

~

~

We integrate on the different coset of the form $a_{i}(1+\varpi^{s+1}O_{F})\subset1+\varpi^{m+1}O_{F}$

~

\[
I_{2}=\underset{1+\varpi^{m+1}O_{F}}{\int}\chi^{*}(N(t)-(1+\pi^{s}\rho))|N(t)-(1+\pi^{s}\rho)|^{s_{1}}dy=\overset{}{\underset{i}{\sum}}B_{i}\underset{a_{i}(1+\varpi^{s+1}O_{F})}{\int}\chi^{*}(N(t)-(1+\pi^{s}\rho))dt
\]

~

Now we show that each integral in the different terms vanish: 

~

\[
I_{2,i}=\underset{a_{i}(1+\varpi^{s+1}O_{F})}{\int}\chi^{*}(N(t)-(1+\pi^{s}\rho))|dt.
\]

We Substitute $u=N(t)$ (note that $N(1+\varpi^{s+1}O_{F})=1+\pi^{s+1}O_{E}$),
by making use of Lemma \ref{pushforword lemma} we get:

\[
I_{2,i}=C_{1}\underset{N(a_{i})(1+\pi^{s+1}O_{E})}{\int}\chi^{*}(u-(1+\pi^{s}\rho))du
\]

~

We substitute $\eta=u-(1+\pi^{s}\rho)$. Since we know that $q^{-s}<|\eta|_{E}\leq q^{-m-1}$
, we deduce that we can present the integral in the following form: 

\[
I_{2,i}=C_{2}\underset{\pi^{k}\zeta_{i}+\pi^{s+1}O_{E}}{\int}\chi^{*}(\eta)d\eta,
\]

where $\zeta_{i}\in O_{F}^{*}$ and $m+1\leq k<s$ . We substitute
again $y=\pi^{-k}\zeta_{i}^{-1}\cdot\eta$ to get:

\[
I_{2,i}=C_{3}\underset{1+\pi^{s+1-k}O_{E}}{\int}\chi^{*}(y)dy=0.
\]

Since $1+\pi^{s+1-k}O_{E}$ contains a non-norm elements, $\chi^{*}$is
a non-trivial character and the integral vanish.

~

~

\textbf{\uline{3.}}

$x=\left(\begin{array}{cc}
\pi^{m} & \varpi\\
\overline{\varpi} & -\pi^{s-m}\rho
\end{array}\right)$~

A similar proof to the previous representative show that $L(x,\chi^{*},s_{1},s_{2})=0.$

~

~

\textbf{\uline{4.}}

$x=\left(\begin{array}{cc}
\pi^{m} & 1\\
1 & 0
\end{array}\right)$

~
\begin{itemize}
\item We prove that $I_{1}=\underset{O_{F}}{\int}\chi^{*}(\pi^{m}+Tr(t))|\pi^{m}+Tr(t)|^{s_{1}}dt=0$
\end{itemize}
We substitute $u=Tr(t)$. We know that (see Section \ref{sub:The-Trace})
$Tr(O_{F})=\pi^{h}O_{F},\,\,\,\frac{s}{2}\leq h\leq\frac{s}{2}+1$
:

~

\[
\underset{O_{F}}{\int}\chi^{*}(\pi^{m}+Tr(t))|\pi^{m}+Tr(t)|^{s_{1}}dt=A_{1}\underset{\pi^{h}O_{E}}{\int}\chi^{*}(\pi^{m}+u)|\pi^{m}+u|^{s_{1}}du.
\]

~

Since $m<h$ the absolute value is constant. We have:

~

\[
I_{1}=A_{2}\underset{\pi^{\frac{s}{2}}O_{E}}{\int}\chi^{*}(\pi^{m}+t)dt=A_{3}\underset{1+\pi^{h-m}O_{E}}{\int}\chi^{*}(t)dt=0
\]

Since $h-m<s+1$ we have that the group $1+\pi^{h-m}O_{E}$ contains
a non norm elements and hence $\chi^{*}$ is a non-trivial character
and the integral vanishes.

~

~
\begin{itemize}
\item We show that $I_{2}=0.$ We have already shown in Section \ref{6.2.1 matrix 4}
\end{itemize}
\[
I_{2}=\underset{\varpi O_{F}}{\int}\chi^{*}(\pi^{m}N(y)+Tr(y))|\pi^{m}N(y)+Tr(y)|^{s_{1}}dy]=B_{1}\underset{1+\varpi^{2m+1}O_{F}}{\int}\chi^{*}(N(t)-1)|N(t)-1|^{s_{1}}dt.
\]

For every $2m+1\leq k$ we show that the integral vanish on the set
$N^{-1}(1+\pi^{k}O_{E}^{*})\cap(1+\varpi^{2m+1}O_{F})$ :

~

For $k<s+1$ the set $N^{-1}(1+\pi^{k}O_{E}^{*})\cap(1+\varpi^{2m+1}O_{F})$
can be represented as a union of cosets:
\[
\underset{}{N^{-1}(1+\pi^{k}O_{E}^{*})\cap(1+\varpi^{2m+1}O_{F})=\underset{i}{\cup}}a_{i}[1+\varpi^{s+1}O_{F}].
\]

On each coset the absolute value is constant ans so:

~

\[
\underset{N^{-1}(1+\pi^{k}O_{E}^{*})\cap(1+\varpi^{2m+1}O_{F})}{\int}\chi^{*}(N(t)-1)|N(t)-1|^{s_{1}}dt=\underset{i}{\sum}C_{i}\underset{a_{i}[1+\varpi^{s+1}O_{F}]}{\int}\chi^{*}(N(t)-1)dt
\]

~

~

We substitute $N(t)-1=u$ (use Lemma \ref{pushforword lemma} and
Section \ref{sub:The-Norm}) , then on each term in the sum:

\[
\underset{a_{i}[1+\varpi^{s+1}O_{F}]}{\int}\chi^{*}(N(t)-1)dt=A_{i}\underset{N(a_{i})[1+\pi^{s+1}O_{E}]-1}{\int}\chi^{*}(u)du
\]

~

But we know that $|u|_{E}=q^{-k}$ and so we can represent the domain
of integration as:

~

\[
N(a_{i})(1+\pi^{s+1}O_{E})-1=\pi^{k}\eta_{i}+\pi^{s+1}O_{E},\,\,\,\eta_{i}\in O_{E}^{*}
\]

~

We have:

~

\[
\underset{N(a_{i})[1+\pi^{s+1}O_{E}]}{\int}\chi^{*}(u-1)du=\underset{\pi^{k}\eta_{i}+\pi^{s+1}O_{E}}{\int}\chi^{*}(u)du
\]

~

We substitute $s=\pi^{-k}\eta_{i}^{-1}u$ and get:

~

\[
\underset{\pi^{k}\eta_{i}+\pi^{s+1}O_{E}}{\int}\chi^{*}(u)du=\underset{1+\pi^{s+1-k}O_{E}}{\int}\chi^{*}(s)ds=0
\]

~

Since $1+\pi^{s+1-k}O_{E}$ contains a non-norm elements $\chi^{*}$
is a non-trivial character and the integral vanish.

~

~

For $s+1\leq k$ we can represent $N^{-1}(1+\pi^{k}O_{E}^{*})\cap1+\varpi^{2m+1}O_{F}=\underset{i}{\cup}a_{i}(1+\varpi^{k+1}O_{F})$ 

~

Now we integrate on coset of the form $a_{i}(1+\varpi^{k+1}O_{F})$,
similarly the absolute value is constant:

~

\[
\underset{N^{-1}(1+\pi^{k}O_{E}^{*})\cap(1+\varpi^{2m+1}O_{F})}{\int}\chi^{*}(N(t)-1)|N(t)-1|^{s_{1}}dt=\underset{i}{\sum}D_{i}\underset{a_{i}[1+\varpi^{k+1}O_{F}]}{\int}\chi^{*}(N(t)-1)dt.
\]

~

We substitute $N(t)-1=u.$ On each term :

\[
\underset{a_{i}[1+\varpi^{k+1}O_{F}]}{\int}\chi^{*}(N(t)-1)dt=\underset{N(a_{i})(1+\pi^{k+1}O_{E})}{\int}\chi^{*}(u)du
\]

~

But $|u|_{E}=q^{-k}$ so one can represent the domain of integration
as:
\[
N(a_{i})(1+\pi^{k+1}O_{E})=\pi^{k}\eta_{i}+\pi^{k+1}O_{E},\,\,\,\eta_{i}\in O_{E}^{*}
\]

We have by making the substitution $y=\pi^{-k}\eta_{i}^{-1}$ :

\[
\underset{\pi^{k}\eta_{i}+\pi^{k+1}O_{E}}{\int}\chi^{*}(u)du=D\underset{1+\pi O_{E}}{\int}\chi^{*}(y)dy=0.
\]

We showed that $\underset{N^{-1}(1+\pi^{k}O_{E}^{*})\cap(1+\varpi^{2m+1}O_{F})}{\int}\chi^{*}(N(t)-1)|N(t)-1|^{s_{1}}dt=0$
for every $2m+1\leq k$ , and so :

\[
\underset{1+\varpi^{2m+1}O_{F}}{\int}\chi^{*}(N(t)-1)|N(t)-1|^{s_{1}}dt=0\Rightarrow I_{2}=0.
\]

~

\textbf{\uline{5.}}

~

$x=\left(\begin{array}{cc}
\pi^{m} & \varpi\\
\overline{\varpi} & 0
\end{array}\right)$- 

~

A similar proof to the previous will show that $L(x,\chi^{*},s_{1},s_{2})=0.$ 

~

\subsubsection{\label{sub:RP-representative:}RP representative:}

~

$x=\left(\begin{array}{cc}
\pi^{\frac{s}{2}} & 1\\
1 & -\pi^{\frac{s}{2}}\rho
\end{array}\right)$ (RP)

~

\[
L(x,1,\chi_{2},z)=\frac{\chi_{2}(-\Delta)}{q^{-1}+1}\underset{O_{F}}{[\int}\chi^{*}(\pi^{\frac{s}{2}}+Tr(t)-\pi^{\frac{s}{2}}\rho N(t))|\pi^{\frac{s}{2}}+Tr(t)-\pi^{\frac{s}{2}}\rho N(t)|^{s_{1}}dt+
\]

\[
\underset{\varpi O}{\int}\chi^{*}(\pi^{\frac{s}{2}}N(y)+Tr(y)-\pi^{\frac{s}{2}}\rho)|\pi^{\frac{s}{2}}N(y)+Tr(y)-\pi^{\frac{s}{2}}\rho|^{s_{1}}dy].
\]

\begin{itemize}
\item We show that $I_{1}=0$:
\end{itemize}
\begin{equation}
I_{1}=\underset{O_{F}}{\int}\chi^{*}(\pi^{\frac{s}{2}}+Tr(t)-\pi^{\frac{s}{2}}\rho N(t))|\pi^{\frac{s}{2}}+Tr(t)-\pi^{\frac{s}{2}}\rho N(t)|^{s_{1}}dt\label{eq:8.1}
\end{equation}

~

Setting the following identity to \eqref{eq:8.1}:

\[
-\rho\pi^{\frac{s}{2}}N(t)+Tr(u)=(\varpi^{\frac{s}{2}}\alpha t+\overline{\alpha^{-1}\varpi^{-\frac{s}{2}}})(\overline{\varpi^{\frac{s}{2}}\alpha t}+\alpha^{-1}\varpi^{-\frac{s}{2}})+\frac{\pi^{-\frac{s}{2}}}{\rho}=-\frac{\pi^{-\frac{s}{2}}}{\rho}N(\pi^{s}\rho t+1)+\frac{\pi^{-\frac{s}{2}}}{\rho}
\]

~

\[
I_{1}=\underset{O_{F}}{\int}\chi^{*}(\pi^{\frac{s}{2}}-\frac{\pi^{-\frac{s}{2}}}{\rho}N(\pi^{s}\rho t+1)+\frac{\pi^{-\frac{s}{2}}}{\rho})|\pi^{\frac{s}{2}}-\frac{\pi^{-\frac{s}{2}}}{\rho}N(\pi^{s}\rho t+1)+\frac{\pi^{-\frac{s}{2}}}{\rho}|^{s_{1}}dt
\]
~

Simplifying this expression, we have:

\[
I_{1}=A\underset{O_{F}}{\int}\chi^{*}(1+\pi^{s}\rho-N(\pi^{s}\rho t+1))|1+\pi^{s}\rho-N(\pi^{s}\rho u+1)|^{s_{1}}dt.
\]

We Substitute $u=1+\pi^{s}\rho t$:

\[
I_{1}=A\underset{1+\pi^{s}O_{F}}{\int}\chi^{*}(1+\pi^{s}\rho-N(u))|1+\pi^{s}\rho-N(u)|^{s_{1}}dt.
\]

~

Since $N(u)\in1+\pi^{s}O_{E}$ and $1+\pi^{s}\rho\notin N(F^{*})$
we have that $|N(u)-(1+\pi^{s}\rho)|=q^{-2\cdot s\cdot s_{1}}$ :

~

\[
I_{1}=B\underset{1+\pi^{s}O_{F}}{\int}\chi^{*}(1+\pi^{s}\rho-N(u))du.
\]

~

Substitute $\nu=N(u)$ :

\[
I_{1}=C\underset{1+\pi^{2s}O_{E}}{\int}\chi^{*}((1+\pi^{s}\rho)-\nu)d\nu.
\]

~

Substitute $y=1+\pi^{s}\rho-\nu$ :

\[
I_{1}=D\underset{\pi^{s}\rho+\pi^{2s}O_{E}}{\int}\chi^{*}(y)dy.
\]

Substitute $t=\pi^{-s}\rho^{-1}y$:

\[
I_{1}=E\underset{1+\pi^{s}O_{E}}{\int}\chi^{*}(y)dy=0.
\]

\begin{itemize}
\item We Show that $I_{2}=0$, since the absolute value in the integral
is constant, we have that:
\end{itemize}
\[
I_{2}=\underset{\varpi O}{\int}\chi^{*}(\pi^{\frac{s}{2}}N(y)+Tr(y)-\pi^{\frac{s}{2}}\rho)|\pi^{\frac{s}{2}}N(y)+Tr(y)-\pi^{\frac{s}{2}}\rho|^{s_{1}}dy=
\]

\[
A_{2}\underset{\varpi O_{F}}{\int}\chi^{*}(\pi^{\frac{s}{2}}N(y)+Tr(y)-\pi^{\frac{s}{2}}\rho)dy
\]

We use the identity:

\[
\pi^{\frac{s}{2}}N(y)+Tr(y)=(\varpi^{-\frac{s}{2}}+\varpi^{\frac{s}{2}}y)(\overline{\varpi^{-\frac{s}{2}}}+\overline{\varpi^{\frac{s}{2}}}\overline{y})-\pi^{-\frac{s}{2}}=N(\varpi^{-\frac{s}{2}}+\varpi^{\frac{s}{2}}y)-\pi^{-\frac{s}{2}}.
\]

Setting it into the integral we get:

~

\[
\underset{\varpi O_{F}}{\int}\chi^{*}(\pi^{\frac{s}{2}}N(y)+Tr(y)-\pi^{\frac{s}{2}}\rho)dy=\underset{\varpi O_{F}}{\int}\chi^{*}(N(\varpi^{-\frac{s}{2}}+\varpi^{\frac{s}{2}}y)-\pi^{-\frac{s}{2}}-\pi^{\frac{s}{2}}\rho)dy=
\]

\[
\underset{\varpi O_{F}}{\int}\chi^{*}(N(1+\varpi^{s}y)-(1+\pi^{s}\rho))dy
\]

We substitute $t=1+\varpi^{s}y$ to get:

\[
I_{2}=A_{3}\underset{1+\varpi^{s+1}O_{F}}{\int}\chi^{*}(N(t)-(1+\pi^{s}\rho))dx
\]

We substitute $u=N(x)$ and use Lemma \ref{pushforword lemma} :

\begin{equation}
I_{2}=A_{4}\underset{1+\pi^{s+1}O_{E}}{\int}\chi^{*}(u-(1+\pi^{s}\rho)du
\end{equation}

We substitute $\nu=u-(1+\pi^{s}\rho)$

\[
I_{2}=A_{5}\underset{-\pi^{s}\rho+\pi^{s+1}O_{E}}{\int}\chi^{*}(\nu)d\nu.
\]

We substitute $-\pi^{s}\rho\cdot\zeta=\nu$ to get:

\[
I_{2}=A_{6}\underset{1+\pi O_{E}}{\int}\chi^{*}(\zeta)d\zeta=0.
\]

Since $1+\pi O_{E}$ contains a non-norm, $\chi^{*}$ is a non-trivial
character and the integral vanish.

~

\subsubsection{RU representative:}

$x\sim\left(\begin{array}{cc}
\pi^{\frac{s+1}{2}} & \varpi\\
\overline{\varpi} & \pi^{\frac{s+1}{2}}\rho
\end{array}\right)$ 

~

Showing that $L(x,\chi^{*},s_{1},s_{2})=0$ is similar to the RP equivalent
representative in Subsection \ref{sub:RP-representative:}

\subsection{Calculating $L(\chi^{*},\chi_{2},x,z)$ on the diagonal representatives.}

~

$x\sim\left(\begin{array}{cc}
\pi^{\lambda_{1}}\epsilon_{1}\\
 & \pi^{\lambda_{2}}\epsilon_{2}
\end{array}\right)$

~

By Lemma \ref{formula lemma}:

\[
L(x,\chi^{*},\chi_{2},s_{1},s_{2})=\frac{q^{\frac{\lambda_{2}-\lambda_{1}}{2}}q^{\lambda_{1}z_{2}+\lambda_{2}z_{1}}\chi_{2}(\epsilon_{1}\epsilon_{2})\chi^{*}(\epsilon_{2})}{q^{-1}+1}[I_{1}+I_{2}].
\]

Where:

\[
\begin{array}{c}
I_{1}=\underset{x\in O_{F}}{\int}\chi^{*}(\pi^{\lambda_{1}-\lambda_{2}}\frac{\epsilon_{1}}{\epsilon_{2}}+N(x))|\pi^{\lambda_{1}-\lambda_{2}}\frac{\epsilon_{1}}{\epsilon_{2}}+N(x)|^{s_{1}}dx\\
I_{2}=\underset{y\in\varpi O_{F}}{\int}\chi^{*}(1+\frac{\epsilon_{1}}{\epsilon_{2}}\pi^{\lambda_{1}-\lambda_{2}}N(y))|1+\frac{\epsilon_{1}}{\epsilon_{2}}\pi^{\lambda_{1}-\lambda_{2}}N(y)|^{s_{1}}dy.
\end{array}
\]

~

~

\subsubsection{\label{sub:The-representatives-with lambda1-lambda2<s}The representatives
{\small with} $\lambda_{1}-\lambda_{2}<s$}

~

~

Note that 
\[
I_{2}=C\underset{y\in\varpi O_{F}}{\int}\chi^{*}(1+\frac{\epsilon_{1}}{\epsilon_{2}}\pi^{\lambda_{1}-\lambda_{2}}N(y))dy=D\underset{y\in O_{F}}{\int}\chi^{*}(1+\frac{\epsilon_{1}}{\epsilon_{2}}\pi^{\lambda_{1}-\lambda_{2}+1}N(y))dy
\]

By Lemma  \ref{charchter lemma 1} this integral vanishes.

~

Showing $I_{1}=0:$

~

~

~

~

\[
I_{1}=\underset{t\in O_{F}}{\int}\chi^{*}(\pi^{\lambda_{1}}\epsilon_{1}+\pi^{\lambda_{2}}\epsilon_{2}N(t))|\pi^{\lambda_{1}}\epsilon_{1}+\pi^{\lambda_{2}}\epsilon_{2}N(t)|^{s_{1}}dt=
\]

\[
q^{-2\lambda_{2}s_{1}}\underset{x\in O_{F}}{\int}\chi^{*}(\pi^{\lambda_{1}-\lambda_{2}}\frac{\epsilon_{1}}{\epsilon_{2}}+N(t))|\pi^{\lambda_{1}-\lambda_{2}}\frac{\epsilon_{1}}{\epsilon_{2}}+N(t)|^{s_{1}}dt
\]

~

~

This is most complicated and tricky, we prove it with careful steps:

~

\textbf{\large Step 1}: 

We Show that If $0\leq m<s$ 
\[
I_{1}=A\underset{x\in O_{F}}{\int}\chi^{*}(\pi^{m}\frac{\epsilon_{1}}{\epsilon_{2}}+N(t))|\pi^{m}\frac{\epsilon_{1}}{\epsilon_{2}}+N(t)|^{s_{1}}dx=0
\]

~

Decompose the integral to $\underset{i}{\cup}\varpi^{i}O_{F}^{*}$
:

\[
I_{1}=A\underset{i}{\sum}\underset{\varpi^{i}O_{F}^{*}}{\int}\chi^{*}(\pi^{m}\frac{\epsilon_{1}}{\epsilon_{2}}+N(t))|\pi^{m}\frac{\epsilon_{1}}{\epsilon_{2}}+N(t)|^{s_{1}}dt
\]

For $i\leq m$ the absolute value is constant, so on each term. We
substitute $t=u^{-1}$ to get:

\[
\underset{O_{F}^{*}}{\int}\chi^{*}(\pi^{m-i}\frac{\epsilon_{1}}{\epsilon_{2}}+N(t))dt=B\underset{O_{F}^{*}}{\int}\chi^{*}(1+\pi^{m-i}\frac{\epsilon_{1}}{\epsilon_{2}}N(u))du
\]

By Lemma  \ref{charchter lemma 2} this integral vanishes.

~

~

So we have: 
\[
I_{1}=A\underset{i\geq m}{\sum}\underset{\varpi^{i}O_{F}^{*}}{\int}\chi^{*}(\pi^{m}\frac{\epsilon_{1}}{\epsilon_{2}}+N(t))|\pi^{m}\frac{\epsilon_{1}}{\epsilon_{2}}+N(t)|^{s_{1}}dt
\]

\[
=\underset{\varpi^{m}O_{F}}{\int}\chi^{*}(\pi^{m}\frac{\epsilon_{1}}{\epsilon_{2}}+N(t))|\pi^{m}\frac{\epsilon_{1}}{\epsilon_{2}}+N(t)|^{s_{1}}dt
\]

~

On the space $\varpi^{m+1}O_{F}$ the absolute value is constant again,
Substitute $u=\varpi^{-m-1}t$:

~

\[
\underset{\varpi^{m+1}O_{F}}{\int}\chi^{*}(\pi^{m}\frac{\epsilon_{1}}{\epsilon_{2}}+N(t))dt=C\underset{O_{F}}{\int}\chi^{*}(\frac{\epsilon_{1}}{\epsilon_{2}}+\pi N(u))du=0
\]
 By Lemma \ref{charchter lemma 1}

~

So we conclude that: 
\[
I_{1}=A\underset{\varpi^{m}O_{F}^{*}}{\int}\chi^{*}(\pi^{m}\frac{\epsilon_{1}}{\epsilon_{2}}+N(t))|\pi^{m}\frac{\epsilon_{1}}{\epsilon_{2}}+N(t)|^{s_{1}}dt
\]
.

~

Substitute $u=\varpi^{-m}t$ to get:

~

\begin{equation}
I_{1}=C\underset{x\in O_{F}^{*}}{\int}\chi^{*}(\frac{\epsilon_{1}}{\epsilon_{2}}+N(t))|\frac{\epsilon_{1}}{\epsilon_{2}}+N(t)|^{s_{1}}dt\label{eq:8.3}
\end{equation}

~

Observe that the integrand is not a locally constant function.

~

~

\textbf{\large Step 2: }{\large \par}

~

We compute the integral in \eqref{eq:8.3} on cosets of the group:
$\frac{O_{F}^{*}}{1+\varpi O_{F}}$

~

\[
\underset{O_{F}^{*}}{\int}\chi^{*}(\frac{\epsilon_{1}}{\epsilon_{2}}+N(t))|\frac{\epsilon_{1}}{\epsilon_{2}}+N(t)|^{s_{1}}dt=\underset{i}{\Sigma}\underset{a_{i}(1+\varpi O_{F})}{\int}\chi^{*}(\frac{\epsilon_{1}}{\epsilon_{2}}+N(t))|\frac{\epsilon_{1}}{\epsilon_{2}}+N(t)|^{s_{1}}dt.
\]

We substitute on each term $u=a_{i}^{-1}t$ . 

Note that the integrand is invariant to the substitution, it follows
that:

~

\[
I_{1}=D\underset{1+\varpi O_{F}}{\int}\chi^{*}(\frac{\epsilon_{1}}{\epsilon_{2}}+N(t))|\frac{\epsilon_{1}}{\epsilon_{2}}+N(t)|^{s_{1}}dt
\]

~

Let $\zeta\in1+\varpi O_{F}$. we prove: 
\[
J=\underset{1+\varpi O_{F}}{\int}\chi^{*}(\zeta+N(t))|\zeta+N(t)|^{s_{1}}dx=0.
\]

~

We decompose the integration on spaces such that the absolute value
would be constant.

~

We have that (up to a measure zero subset): 

\[
N(1+\varpi O_{F})\subseteq1+\pi O_{E}=\underset{j\geq1}{\cup}(1+\pi^{j}O_{E}^{*}).
\]

Observe that:

\[
\underset{j\geq1}{\cup}(1+\pi^{j}O_{E}^{*})=\underset{j\geq1}{\cup}(-\zeta+\pi^{j}O_{E}^{*}).
\]

~

Now we calculate $J$ on the spaces $N^{-1}(-\zeta+\pi^{j}O_{E}^{*})$:

~
\[
J=\underset{j=1}{\overset{\infty}{\sum}}\underset{N^{-1}(1+\pi^{j}O_{E}^{*})}{\int}\chi^{*}(\zeta+N(t))|\zeta+N(t|^{s_{1}}dt
\]

If $0<j\leq s$ , the inverse image $N^{-1}(-\zeta+\pi^{j}O_{E}^{*})$
is a disjoint union of coset $\underset{k}{\cup}a_{k,j}(1+\varpi^{s+1}O_{F})$
.

For $s<j$, $N^{-1}(-\zeta+\pi^{j}O_{E}^{*})$ can be represented
as union of $b_{k,j}(1+\varpi^{j+1}O_{F})$ 

We compute the integral on each coset:

~

\[
J=\underset{j=1}{\overset{s}{\sum}}\underset{N^{-1}(1+\varpi O_{F})}{\int}\chi^{*}(\zeta+N(t))|\zeta+N(t)|^{s_{1}}dt+
\]

\[
\underset{j>s}{\overset{}{\sum}}\underset{N^{-1}(-\zeta+\pi^{j}O_{E}^{*})}{\int}\chi^{*}(\zeta+N(t))|\zeta+N(t)|^{s_{1}}dt=
\]

\[
\underset{j=1}{\overset{s}{\sum}}\underset{k}{\sum}\underset{a_{k,j}(1+\varpi^{s+1}O_{F})}{\int}\chi^{*}(\zeta+N(t))|\zeta+N(t)|^{s_{1}}dt+
\]

\[
\underset{j>s}{\overset{}{\sum}}\underset{k}{\sum}\underset{b_{k,j}(1+\varpi^{j+1}O_{F})}{\int}\chi^{*}(\zeta+N(t))|\zeta+N(t)|^{s_{1}}dt.
\]

~

~

Since for every $t\in a_{k,j}(1+\varpi^{s+1}O_{F})$ we have that
$|\zeta+N(t)|_{E}=q^{-j}$ , after the substitution: $\nu=\zeta+N(t)$
, each one of the domains of integration $a_{k,j}(1+\pi^{s+1}O_{F})$
can be represented :

\[
\pi^{j}\eta_{i,j}+\pi^{s+1}O_{E},\,\,\,\eta_{i,j}\in O_{F}^{*}
\]

~

And so we have that:

\[
\underset{a_{k,j}(1+\varpi^{s+1}O_{F})}{\int}\chi^{*}(\zeta+N(t))|\zeta+N(t)|^{s_{1}}dt=D_{i,j}\underset{\pi^{j}\eta_{i,j}+\pi^{s+1}O_{E}}{\int}\chi^{*}(\nu)|\nu|^{s_{1}}d\nu.
\]

~

Substitute $u=\pi^{-j}\eta_{i,j}^{-1}\nu$ to get:

~

\[
\underset{\pi^{j}\eta_{i,j}+\pi^{s+1}O_{E}}{\int}\chi^{*}(\nu)|\nu|^{s_{1}}d\nu=E_{i,j}\underset{x\in1+\pi^{s+1-j}O_{E}}{\int}\chi^{*}(t)dt=0
\]
.

Because $\chi^{*}$ is a non trivial character on $1+\pi^{s+1-j}O_{E}$.

~

A similar trick will show that the integral $\underset{b_{k,j}(1+\varpi^{j+1}O_{F})}{\int}\chi^{*}(\zeta+N(t))|\zeta+N(t)|^{s_{1}}dt=0.$ 

~

We deduce that: $J=0\Longrightarrow I_{1}=0$

~

In conclusion $I_{1}=I_{2}=0\Longrightarrow L(x,\chi^{*},\chi_{2},z)=0$.

~

~

\subsubsection{The representative with $\lambda_{1}-\lambda_{2}=s${\large :}}

Suppose $\lambda_{1}-\lambda_{2}=s$

~

In this case the 
\[
I_{2}=\underset{y\in\varpi O_{F}}{\int}\chi^{*}(1+\frac{\epsilon_{1}}{\epsilon_{2}}\pi^{\lambda_{1}-\lambda_{2}}N(y))|1+\frac{\epsilon_{1}}{\epsilon_{2}}\pi^{\lambda_{1}-\lambda_{2}}N(y)|^{s_{1}}dy\neq0
\]
but we have: 
\[
I_{1}+I_{2}=I_{1}-\underset{O_{F}^{*}}{\int}\chi^{*}(\pi^{\lambda_{1}-\lambda_{2}}\frac{\epsilon_{1}}{\epsilon_{2}}+N(t))dt+\underset{O_{F}^{*}}{\int}\chi^{*}(\pi^{\lambda_{1}-\lambda_{2}}\frac{\epsilon_{1}}{\epsilon_{2}}+N(t))dt+I_{2}
\]

~

Note that by substituting $t\mapsto t^{-1}$, we have:: 
\[
\underset{O_{F}^{*}}{\int}\chi^{*}(\pi^{\lambda_{1}-\lambda_{2}}\frac{\epsilon_{1}}{\epsilon_{2}}+N(t))dt=\underset{O_{F}^{*}}{\int}\chi^{*}(\pi^{\lambda_{1}-\lambda_{2}}\frac{\epsilon_{1}}{\epsilon_{2}}N(t)+1)dt.
\]

~

By Lemma \ref{charchter lemma 1}, the sum of the integrals vanish:

\[
\underset{O_{F}^{*}}{\int}\chi^{*}(\pi^{\lambda_{1}-\lambda_{2}}\frac{\epsilon_{1}}{\epsilon_{2}}N(t)+1)dt+I_{2}=\underset{O_{F}}{\int}\chi^{*}(1+\pi^{\lambda_{1}-\lambda_{2}}\frac{\epsilon_{1}}{\epsilon_{2}}N(t))dt=0
\]

~

Hence: 

\[
I_{1}+I_{2}=I_{1}-\underset{O_{F}^{*}}{\int}\chi^{*}(\pi^{\lambda_{1}-\lambda_{2}}\frac{\epsilon_{1}}{\epsilon_{2}}+N(t))dt.
\]

But:

\[
I_{1}-\underset{O_{F}^{*}}{\int}\chi^{*}(\pi^{\lambda_{1}-\lambda_{2}}\frac{\epsilon_{1}}{\epsilon_{2}}+N(t))dt=C\underset{\varpi O_{F}}{\int}|\pi^{\lambda_{1}-\lambda_{2}}\frac{\epsilon_{1}}{\epsilon_{2}}+N(t)|\chi^{*}(\pi^{\lambda_{1}-\lambda_{2}}\frac{\epsilon_{1}}{\epsilon_{2}}+N(t))dt=
\]

\[
D\underset{O_{F}}{\int}|\pi^{\lambda_{1}-\lambda_{2}-1}\frac{\epsilon_{1}}{\epsilon_{2}}+N(t)|\chi^{*}(\pi^{\lambda_{1}-\lambda_{2}-1}\frac{\epsilon_{1}}{\epsilon_{2}}+N(t))dt
\]
The last integral vanish by Subsection \ref{sub:The-representatives-with lambda1-lambda2<s}
($\lambda_{1}-\lambda_{2}<s$).

~

~

\subsection{The diagonal representatives with $\lambda_{1}-\lambda_{2}>s$}

~

This is the only case where the spherical function do not vanish and
$\chi_{1}=\chi^{*}$:

~

\[
L(x,\chi^{*},\chi_{2},s_{1},s_{2})=\frac{q^{\frac{\lambda_{2}-\lambda_{1}}{2}}q^{\lambda_{1}z_{2}+\lambda_{2}z_{1}}\chi_{2}(\epsilon_{1}\epsilon_{2})\chi^{*}(\epsilon_{2})}{q^{-1}+1}[I_{1}+I_{2}]
\]

~

\[
I_{1}=\underset{O_{F}}{\int}\chi^{*}(\pi^{\lambda_{1}-\lambda_{2}}\frac{\epsilon_{1}}{\epsilon_{2}}+N(t))|\pi^{\lambda_{1}-\lambda_{2}}\frac{\epsilon_{1}}{\epsilon_{2}}+N(t)|^{s_{1}}dt
\]

~

\[
I_{2}=\underset{y\in\varpi O_{F}}{\int}\chi^{*}(1+\frac{\epsilon_{1}}{\epsilon_{2}}\pi^{\lambda_{1}-\lambda_{2}}N(y))|1+\frac{\epsilon_{1}}{\epsilon_{2}}\pi^{\lambda_{1}-\lambda_{2}}N(y)|^{s_{1}}dy
\]

~

~

Since $1+\frac{\epsilon_{1}}{\epsilon_{2}}\pi^{\lambda_{1}-\lambda_{2}}N(y)$
is a norm if $y\in\varpi O_{F}$ it could be shown easily that $I_{2}=q^{-1}.$

~

We integrate $I_{1}$ on the spaces $\underset{i\geq0}{\cup}\varpi^{i}O_{F}^{*}$
:

~
\begin{itemize}
\item If $\lambda_{1}-\lambda_{2}-s>i$ and $t\in\varpi^{i}O_{F}^{*}$ then
$\pi^{\lambda_{1}-\lambda_{2}}\frac{\epsilon_{1}}{\epsilon_{2}}+N(t)$
is a norm and 
\end{itemize}
~

\[
\underset{\varpi^{i}O_{F}^{*}}{\int}\chi^{*}(\pi^{\lambda_{1}-\lambda_{2}}\frac{\epsilon_{1}}{\epsilon_{2}}+N(t))|\pi^{\lambda_{1}-\lambda_{2}}\frac{\epsilon_{1}}{\epsilon_{2}}+N(t)|^{s_{1}}dx=q^{-2is_{1}-i}(1-q^{-1})
\]
.

~

~

So 
\[
I_{1}=\underset{j=0}{\overset{\lambda_{1}-\lambda_{2}-s-1}{\sum}}(1-q^{-1})q^{-2js_{1}-j}+\underset{i\geq\lambda_{1}-\lambda_{2}-s}{\sum}\underset{\varpi^{i}O_{F}^{*}}{\int}\chi^{*}(\pi^{\lambda_{1}-\lambda_{2}}\frac{\epsilon_{1}}{\epsilon_{2}}+N(t))|\pi^{\lambda_{1}-\lambda_{2}}\frac{\epsilon_{1}}{\epsilon_{2}}+N(t)|^{s_{1}}dt
\]

~

~
\begin{itemize}
\item For $i=\lambda_{1}-\lambda_{2}-s$, we have that if $t\in\varpi^{i}O_{F}^{*}$
then $|\pi^{\lambda_{1}-\lambda_{2}}\frac{\epsilon_{1}}{\epsilon_{2}}+N(t)|^{s_{1}}=q^{-2i\cdot s_{1}}$
:
\end{itemize}
~

~

\[
\underset{x\in\varpi^{i}O_{F}^{*}}{\int}\chi^{*}(\pi^{\lambda_{1}-\lambda_{2}}\frac{\epsilon_{1}}{\epsilon_{2}}+N(t))|\pi^{\lambda_{1}-\lambda_{2}}\frac{\epsilon_{1}}{\epsilon_{2}}+N(t)|^{s_{1}}dt=q^{-2i\cdot s_{1}}\underset{x\in\varpi^{i}O_{F}^{*}}{\int}\chi^{*}(\pi^{\lambda_{1}-\lambda_{2}}\frac{\epsilon_{1}}{\epsilon_{2}}+N(t))dt
\]

~

Substituting $u=\varpi^{i}t^{-1}$ , we get:

~

\begin{equation}
q^{-2i\cdot s_{1}}\underset{x\in\varpi^{i}O_{F}^{*}}{\int}\chi^{*}(\pi^{\lambda_{1}-\lambda_{2}}\frac{\epsilon_{1}}{\epsilon_{2}}+N(t))dt=q^{-2i\cdot s_{1}-i}\underset{O_{F}^{*}}{\int}\chi^{*}(1+\pi^{s}\frac{\epsilon_{1}}{\epsilon_{2}}N(u))du.\label{eq:8.1-2}
\end{equation}

~

By Lemma \ref{charchter lemma 1}:

~

\begin{equation}
\underset{O_{F}}{\int}\chi^{*}(1+\pi^{s}\frac{\epsilon_{1}}{\epsilon_{2}}N(t))dt=0\Rightarrow\underset{O_{F}^{*}}{\int}\chi^{*}(1+\pi^{s}\frac{\epsilon_{1}}{\epsilon_{2}}N(t))dt=-\underset{\varpi O}{\int}\chi^{*}(1+\pi^{s}\frac{\epsilon_{1}}{\epsilon_{2}}N(t))dt=-q^{-1}\label{eq:8.5}
\end{equation}

~

Substituting \eqref{eq:8.5}to Eq \ref{eq:8.1-2} , we get that for
$i=\lambda_{1}-\lambda_{2}-s$:

\[
\underset{\varpi^{i}O_{F}^{*}}{\int}\chi^{*}(\pi^{\lambda_{1}-\lambda_{2}}\frac{\epsilon_{1}}{\epsilon_{2}}+N(t))|\pi^{\lambda_{1}-\lambda_{2}}\frac{\epsilon_{1}}{\epsilon_{2}}+N(t)|^{s_{1}}dt=-q^{-2i\cdot s_{1}-i-1}.
\]

~

So : 

\[
I_{1}=\underset{j=0}{\overset{\lambda_{1}-\lambda_{2}-s-1}{\sum}}(1-q^{-1})q^{-2js_{1}-j}-q^{-2(\lambda_{1}-\lambda_{2}-s)\cdot s_{1}-(\lambda_{1}-\lambda_{2}-s)-1}+
\]

\[
\underset{\varpi^{\lambda_{1}-\lambda_{2}-s+1}O_{F}}{\int}\chi^{*}(\pi^{\lambda_{1}-\lambda_{2}}\frac{\epsilon_{1}}{\epsilon_{2}}+N(t))|\pi^{\lambda_{1}-\lambda_{2}}\frac{\epsilon_{1}}{\epsilon_{2}}+N(t)|^{s_{1}}dt
\]

~

~

Note that: 

\[
\underset{\varpi^{\lambda_{1}-\lambda_{2}-s+1}O_{F}}{\int}\chi^{*}(\pi^{\lambda_{1}-\lambda_{2}}\frac{\epsilon_{1}}{\epsilon_{2}}+N(t))|\pi^{\lambda_{1}-\lambda_{2}}\frac{\epsilon_{1}}{\epsilon_{2}}+N(t)|^{s_{1}}dt=
\]

\[
C_{1}\underset{O_{F}}{\int}\chi^{*}(\frac{\epsilon_{1}}{\epsilon_{2}}+\pi^{-s+1}N(t))|\frac{\epsilon_{1}}{\epsilon_{2}}+\pi^{-s+1}N(t))|^{s_{1}}dt=
\]

\[
C_{2}\underset{O_{F}}{\int}\chi^{*}(\pi^{s-1}\frac{\epsilon_{1}}{\epsilon_{2}}+N(t))|\frac{\epsilon_{1}}{\epsilon_{2}}\pi^{s-1}+N(t))|^{s_{1}}dt
\]

We apply the case of $\lambda_{1}-\lambda_{2}<s$ to get that:

\[
C_{2}\underset{O_{F}}{\int}\chi^{*}(\pi^{s-1}\frac{\epsilon_{1}}{\epsilon_{2}}+N(t))|\frac{\epsilon_{1}}{\epsilon_{2}}\pi^{s-1}+N(t))|^{s_{1}}dt=0
\]

~

~

Altogether:

~

\[
I_{1}=\underset{O_{F}}{\int}\chi^{*}(\pi^{\lambda_{1}-\lambda_{2}}\frac{\epsilon_{1}}{\epsilon_{2}}+N(t))|\pi^{\lambda_{1}-\lambda_{2}}\frac{\epsilon_{1}}{\epsilon_{2}}+N(t)|^{s_{1}}dt=
\]

\[
\underset{j=0}{\overset{\lambda_{1}-\lambda_{2}-s-1}{\sum}}(1-q^{-1})q^{-2js_{1}-j}]-q^{-2(\lambda_{1}-\lambda_{2}-s)\cdot s_{1}-(\lambda_{1}-\lambda_{2}-s)-1}=
\]

\[
[(1-q^{-1})\frac{1-q^{(\lambda_{1}-\lambda_{2}-s)\cdot(2z_{1}-2z_{2})}}{1-q^{2z_{1}-2z_{2}}}-q^{i(2z_{1}-2z_{2}+2)-1}].
\]

~

Recall that:
\[
L(x,\chi^{*},\chi_{2},s_{1},s_{2})=\frac{q^{\frac{\lambda_{2}-\lambda_{1}}{2}}q^{\lambda_{1}z_{2}+\lambda_{2}z_{1}}\chi_{2}(\epsilon_{1}\epsilon_{2})\chi^{*}(\epsilon_{2})}{q^{-1}+1}I_{1}
\]

~

After simplifying:

~

\[
L(x,\chi^{*},\chi_{2},s_{1},s_{2})=\frac{q^{\frac{\lambda_{2}-\lambda_{1}}{2}}q^{\lambda_{1}z_{2}+\lambda_{2}z_{1}}\chi_{2}(\epsilon_{1}\epsilon_{2})\chi^{*}(\epsilon_{2})}{q^{-1}+1}\times
\]

\[
[q^{-1}+(1-q^{-1})\frac{q^{2z_{2}}-q^{(\lambda_{1}-\lambda_{2}-s)\cdot(2z_{1}-2z_{2})+2z_{2}}}{q^{2z_{2}}-q^{2z_{1}}}-q^{(\lambda_{1}-\lambda_{2}-s)(2z_{1}-2z_{2})-1}]
\]
~

After further simplification:

~

\selectlanguage{american}%
\[
L(x,\chi^{*},\chi_{2},s_{1},s_{2})=\frac{q^{\frac{\lambda_{2}-\lambda_{1}}{2}}q^{\lambda_{1}z_{2}+\lambda_{2}z_{1}}\chi_{2}(\epsilon_{1}\epsilon_{2})\chi^{*}(\epsilon_{2})}{q^{-1}+1}\times
\]

~

\[
[\frac{q^{2z_{1}-1}}{q^{2z_{1}}-q^{2z_{2}}}+\frac{-q^{2z_{2}}+q^{(\lambda_{1}-\lambda_{2}-s)\cdot(2z_{1}-2z_{2})+2z_{2}}}{q^{2z_{1}}-q^{2z_{2}}}+\frac{-q^{(\lambda_{1}-\lambda_{2}-s)(2z_{1}-2z_{2})-1+2z_{1}}}{q^{2z_{1}}-q^{2z_{2}}}]
\]

~

\[
=\frac{q^{\frac{\lambda_{2}-\lambda_{1}}{2}}\chi_{2}(\epsilon_{1}\epsilon_{2})\chi^{*}(\epsilon_{2})q^{2sz_{2}}(q^{2z_{2}}-q^{2z_{1}-1})}{q^{-1}+1}\times\frac{(q^{\lambda_{1}z_{1}+\lambda_{2}z_{2}-2sz_{1}}-q^{\lambda_{1}z_{2}+\lambda_{2}z_{1}-2sz_{2}})}{q^{2z_{1}}-q^{2z_{2}}}
\]

We represent the function in a more symmetric form:

\[
L(x,\chi^{*},\chi_{2},s_{1},s_{2})=\frac{q^{\frac{\lambda_{2}-\lambda_{1}}{2}}\chi^{*}(\epsilon_{2})\chi_{2}(\epsilon_{1}\epsilon_{2})}{1+q^{-1}}q^{2sz_{2}}(q^{2z_{2}}-q^{2z_{1}-1})\times\underset{\sigma\in\Sigma_{2}}{\sum}\sigma(\frac{q^{2\langle(\lambda_{1}-s,\lambda_{2}),z\rangle}}{q^{2z_{1}}-q^{2z_{2}}}).
\]

~

\selectlanguage{english}%
~

\end{document}